\documentclass[leqno]{amsart}

\usepackage{amsmath}
\usepackage{amssymb}
\usepackage{amsfonts}
\usepackage{amsthm}
\usepackage[foot]{amsaddr}
\usepackage{mathtools}
\mathtoolsset{%
above-intertext-sep = -1ex
}

\usepackage[utf8]{inputenc}
\usepackage[T1]{fontenc}

\usepackage[sf,mono=false]{libertine}

\usepackage[sans]{dsfont}

\usepackage[%
cal=cm,
scr=euler,
]
{mathalfa}

\usepackage[dvipsnames,svgnames]{xcolor}
\colorlet{MyBlue}{DodgerBlue!75!Black}
\colorlet{MyGreen}{DarkGreen!95!Black}

\usepackage[font=small,labelfont=bf]{caption}
\usepackage{subfigure}

\usepackage{acronym}
\usepackage{latexsym}
\usepackage{paralist}
\usepackage{wasysym}
\usepackage{xspace}
\usepackage{framed}

\usepackage[numbers,sort&compress]{natbib}

\usepackage{hyperref}
\hypersetup{
colorlinks=true,
linktocpage=true,
pdfstartview=FitH,
breaklinks=true,
pdfpagemode=UseNone,
pageanchor=true,
pdfpagemode=UseOutlines,
plainpages=false,
bookmarksnumbered,
bookmarksopen=false,
bookmarksopenlevel=1,
hypertexnames=true,
pdfhighlight=/O,
urlcolor=MyBlue!60!black,linkcolor=MyBlue!70!black,citecolor=DarkGreen!70!black, 
pdftitle={},
pdfauthor={},
pdfsubject={},
pdfkeywords={},
pdfcreator={pdfLaTeX},
pdfproducer={LaTeX with hyperref}
}
\newcommand{\EMAIL}[1]{\email{\href{mailto:#1}{#1}}}

\numberwithin{equation}{section}  
\usepackage[sort&compress,capitalize,nameinlink]{cleveref}

\crefrangeformat{equation}{\upshape(#3#1#4)\textendash(#5#2#6)}

\newcommand{\dd}{\:d}
\newcommand{\eps}{\epsilon}

\newcommand{\dif}{\dd}

\DeclareMathOperator{\cl}{cl}

\DeclareMathOperator{\dom}{dom}
\DeclareMathOperator{\ran}{ran}

\DeclareMathOperator{\gr}{gr}

\DeclareMathOperator{\Id}{Id}

\DeclareMathOperator{\Zer}{Zer}
\DeclareMathOperator{\Fix}{Fix}

\renewcommand{\emptyset}{\varnothing}
\newcommand{\eqdef}{\triangleq}

\newcommand{\scrH}{\mathcal{H}}

\newcommand{\scrS}{\mathcal{S}}

\renewcommand{\Pr}{P}

\newcommand{\R}{\mathbb{R}}

\DeclareMathOperator{\NC}{\mathsf{N}}

\usepackage{algorithm}
\usepackage{algpseudocode}								

\theoremstyle{plain}
\newtheorem{theorem}{Theorem}

\newtheorem*{corollary*}{Corollary}
\newtheorem{lemma}[theorem]{Lemma}
\newtheorem{proposition}[theorem]{Proposition}

\theoremstyle{definition}
\newtheorem{definition}[theorem]{Definition}
\newtheorem*{definition*}{Definition}
\newtheorem{assumption}[theorem]{Assumption}

\theoremstyle{remark}
\newtheorem{remark}{Remark}
\newtheorem*{remark*}{Remark}
\newtheorem*{notation*}{Notational remark}

\numberwithin{theorem}{section}
\numberwithin{remark}{section}
\numberwithin{example}{section}

\DeclarePairedDelimiter{\abs}{\lvert}{\rvert}

\DeclarePairedDelimiter{\inner}{\langle}{\rangle}
\DeclarePairedDelimiter{\norm}{\lVert}{\rVert}

\begin{document}


\title[Strong convergence in dynamical systems with composite structure]{Inducing strong convergence of trajectories in dynamical systems associated to monotone inclusions with composite structure}

\author[R.I. Bo\c t]{Radu Ioan Bo\c t$^{\star}$}
\address{$^{\star}$\,%
Faculty of Mathematics, University of Vienna, Oskar-Morgenstern-Platz 1, A-1090 Vienna, Austria.}
\EMAIL{radu.bot@univie.ac.at}
\author[S.-M. Grad]{Sorin-Mihai Grad$^{\star}$}
\EMAIL{sorin-mihai.grad@univie.ac.at}
\author[D.~Meier]{Dennis Meier$^{\star}$}
\EMAIL{meierd61@univie.ac.at}
\author[M.Staudigl]{\\Mathias Staudigl$^{\diamond}$}
\address{$^{\diamond}$\,%
Maastricht University, Department of Quantitative Economics, P.O. Box 616, NL\textendash 6200 MD Maastricht, The Netherlands.}
\EMAIL{m.staudigl@maastrichtuniversity.nl}


\begin{abstract}
In this work we investigate dynamical systems designed to approach the solution sets of inclusion problems involving the sum of two maximally monotone operators. Our aim is to design methods which guarantee strong convergence of trajectories towards the minimum norm solution of the underlying monotone inclusion problem. To that end, we investigate in detail the asymptotic behavior of dynamical systems perturbed by a Tikhonov regularization where either the maximally monotone operators themselves, or the vector field of the dynamical system is regularized. In both cases we prove strong convergence of the trajectories towards minimum norm solutions to an underlying monotone inclusion problem, and we illustrate numerically qualitative differences between these two complementary regularization strategies. The so-constructed dynamical systems are either of Krasnoselskii-Mann, of forward-backward type or of forward-backward-forward type, and with the help of injected regularization we demonstrate seminal results on the strong convergence of Hilbert space valued evolutions designed to solve monotone inclusion and equilibrium problems. 
\end{abstract}
\maketitle

\renewcommand{\sharp}{\gamma}
\acresetall
\allowdisplaybreaks

\section{Introduction}
\label{sec:introduction}
In 1974, Bruck showed in \cite{Bru75} that trajectories of the steepest descent system 
\begin{equation}
\dot{x}(t)+\partial \Phi(x(t))\ni 0
\end{equation}
minimize the convex, proper, lower semi-continuous potential $\Phi$ defined on a real Hilbert space $\scrH$. They weakly converge towards a minimum of $\Phi$ and the potential decreases along the trajectory towards its minimal value, provided that $\Phi$ attains its minimum. Subsequently, Baillon and Brezis generalized in \cite{BailBre76} this result to differential inclusions whose drift is a maximally monotone operator $A:\scrH\rightrightarrows 2^{\scrH}$, and dynamics 
\begin{equation}
\dot{x}(t)+A(x(t))\ni 0.
\end{equation}
Baillon provided in \cite{Bai78} an example where the trajectories of the steepest descent system converge weakly but not strongly. A key tool in the study of the convergence of the steepest descent method is the association of Fej\'er monotonicity with the Opial lemma.  In 1996, Attouch and Cominetti coupled in \cite{AC6} approximation methods with the steepest descent system by adding a Tikhonov regularization term 
\begin{equation}\label{eq:AC}
\dot{x}(t)+\partial \Phi(x(t))+\eps(t)x(t)\ni 0.
\end{equation}
The time-varying parameter $\eps(t)$ tends to zero and the potential field $\partial\Phi$ satisfies the usual assumptions for strong existence and uniqueness of trajectories. The striking point of their analysis is the strong convergence of the trajectories when the regularization function $t\mapsto \eps(t)$ tends to zero at a sufficiently slow rate. In particular, $\eps\notin L^{1}(\R_{+};\R)$. Then the strong limit is the point of minimal norm among the minima of the convex function $\Phi$. This is a rather surprising result since we know that if $\eps=0$ we can only expect weak convergence of the induced trajectories under the standard hypotheses, and suddenly with the regularization term the convergence is strong without imposing additional demanding assumptions. These papers were the starting point for a flourishing line of research in which dynamical systems motivated by solving challenging optimization and monotone inclusion problems are studied. The formulation of numerical algorithms as continuous-in-time dynamical systems makes it possible to understand the asymptotic properties of the algorithms by relating them to their descent properties in terms of energy and/or Lyapunov functions, and to derive new numerical algorithms via sophisticated numerical discretization techniques (see, for instance, \cite{ABRT04,AABR02,Wib16,MerSta18,MerSta18b}). This paper follows this line of research. In particular, our main aim in this work is to construct dynamical systems designed to solve Hilbert space valued monotone inclusions of the form 
\begin{equation}\label{eq:MIP}\tag{MIP}
\text{find }x^{\ast}\in\scrH\text{ such that } 0\in Ax^{\ast}+Bx^{\ast},
\end{equation}
where $A : \scrH \rightrightarrows \scrH$ is a maximally monotone operator and $B : \scrH \to \scrH$ a $\beta$-cocoercive (respectively a $({1}/{\beta})$-Lipschitz continuous) operator with $\beta > 0$, such that $\Zer(A+B)$ is nonempty,
and our focus is to design methods which guarantee \emph{strong convergence} of the trajectories towards a solution of \eqref{eq:MIP}. This is a considerable advancement when contrasted with existing methods, where usually only weak convergence of trajectories is to be expected, for the strong one additional demanding hypotheses being imposed. Indeed, departing from the seminal work of Attouch and Cominetti \cite{AC6} a thriving series of papers on dynamical systems for solving monotone inclusions of type \eqref{eq:MIP} emerged, relating continuous-time methods to classical operator splitting iterations. A general overview of this still very active topic is given in \cite{PeySor10}. In \cite{BOL}, Bolte studied the weak convergence of the trajectories of the dynamical system 
\begin{equation}\label{eq:BOL}
\left\{\begin{array}{rl}
\dot{x}(t)+x(t)&=\Pr_{C}(x(t)-\gamma \nabla \Phi(x(t)),\\
x(0)&=x_{0}
\end{array}\right.
\end{equation}
where $\Phi:\scrH\to\R$ is a convex and continuously differentiable function defined on a real Hilbert space $\scrH$, and $C\subseteq\scrH$ is a closed and convex subset with an easy to evaluate orthogonal projector $\Pr_{C}$. Bolte shows that the trajectories of the dynamical system converge weakly to a solution of \eqref{eq:MIP} with $A=\NC_{C}$, the normal cone mapping of $C$, and $B=\nabla\Phi$, which is actually an optimal solution to the optimization problem 
\[
\inf_{x\in C}\Phi(x).
\]
Moreover, in \cite[Section 5]{BOL} a Tikhonov regularization term is added to the differential equation, guaranteeing strong convergence of the trajectories of the perturbed dynamical system. More recently, \cite{AbbAtt15} provided a generalization of Bolte's work where the authors proposed the dynamical system 
\begin{equation}\label{eq:ABB}
\left\{\begin{array}{rl}
\dot{x}(t)+x(t)&=J_{\gamma\partial\Phi}(x(t)-\gamma B(x(t))),\\
x(0)&=x_{0},
\end{array}\right.
\end{equation}
where $\Phi:\scrH\to\R\cup\{+\infty\}$ is a proper, convex and lower semi-continuous function defined on a real Hilbert space $\scrH$, and $B:\scrH\to\scrH$ is a cocoercive operator. 

This projection-differential dynamical system relies on the resolvent operator $J_{\gamma\partial\Phi}\eqdef (\Id+\gamma\partial\Phi)^{-1}$, and reduces to the system \eqref{eq:BOL} when the function $\Phi$ is the indicator function of a closed convex set $C\subseteq\scrH$. It is shown that the trajectories of the dynamical system converge weakly to a solution of the associated \eqref{eq:MIP} in which $A=\partial\Phi$.  In this paper, we continue this line of research and generalize it in two directions. Our first set of results is concerned with dynamical systems of the form \eqref{eq:ABB} involving a nonexpansive mapping. Building on \cite{CPS,BC7}, we perturb such a dynamical system with a Tikhonov regularization term that induces the strong convergence of the thus generated trajectories. This family of dynamical systems is of Krasnoselskii-Mann type whose explicit or implicit numerical discretizations are well studied (see \cite{BC2}). Next, we consider a family of asymptotically autonomous semi-flows derived from operator splitting algorithms. These splitting techniques originate from the theoretical analysis of PDEs, and can be traced back to classical work of \cite{LioSta67,LioMer79}. In this direction, we generalize recent results of \cite{BC7} for dynamical systems of forward-backward type, and \cite{BB8} for dynamical systems of forward-backward-forward type. In both of these papers the strong convergence of the trajectories is guaranteed only under demanding additional hypotheses like strong monotonicity of one of the involved operators. On the other hand, in articles like \cite{AC2, CPS} strong converge of trajectories of dynamical systems involving a single monotone operator or function is achieved by means of a suitable Tikhonov regularization under mild conditions. They motivated us to perturb the mentioned dynamical systems from \cite{BC7, BB8} in a similar manner in order to achieve strong convergence of the trajectories of the resulting Tikhonov regularized dynamical systems under natural
assumptions. To the best of our knowledge the only previous contribution in the literature in this direction is the very recent preprint \cite{PAV}, where a Tikhonov regularized dynamical system involving a nonexpansive operator is investigated, whose trajectories strongly converge towards a fixed point of the latter. All these results are special cases of the analysis provided in this paper. 


In the first part of our paper we deal with a Tikhonov regularized Krasnoselskii-Mann dynamical system and show that its trajectories strongly converge towards a fixed point of the governing nonexpansive operator. Afterwards a modification of this dynamical system inspired by \cite{PAV} is proposed, where the involved operator maps a closed convex set to itself and a similar result is obtained under a different hypothesis. The main result of \cite{PAV} is then recovered as a special case, while another special case concerns a Tikhonov regularized forward-backward dynamical system whose trajectories strongly converge towards a zero of a sum of a maximally monotone operator with a single-valued cocoercive one. Because the regularization term is applied to the whole differential equation, we speak in this case of an outer Tikhonov regularized forward-backward dynamical system. In the next section another forward-backward dynamical system, this time with dynamic stepsizes and an inner Tikhonov regularization of the single-valued operator, is investigated and we show the strong convergence of its trajectories towards the minimum norm zero of a similar sum of operators. Afterwards we consider an implicit forward-backward-forward dynamical system with a similar inner Tikhonov regularization of the involved single-valued operator, whose trajectories strongly converge towards the minimum norm zero of a sum of a maximally monotone operator with a single-valued Lipschitz continuous one. In order to illustrate the theoretical results we present some numerical experiments as well, which shed some light on the role of the regularization parameter on the long-run behavior of trajectories. 

\section{Setup and preliminaries}
\label{sec:prelims}

We collect in this section some general concepts from variational and functional analysis. We follow standard notation, as developed in \cite{BC2}.
Let $\scrH$ be a real Hilbert space. A set-valued operator $M:\scrH\rightrightarrows\scrH$ maps points in $\scrH$ to subsets of $\scrH$. We denote by
\begin{align*}
\dom(M)&\eqdef\{x\in\scrH\vert Mx\neq\emptyset\},\\
\ran(M)&\eqdef\{y\in\scrH\vert(\exists x\in\scrH):y\in Mx\},\\
\gr(M)&\eqdef\{(x,y)\in\scrH\times\scrH\vert y\in Mx\},\\
\Zer(M)&\eqdef\{x\in\scrH\vert 0\in Mx\},
\end{align*}
its \textit{domain}, \textit{range}, \textit{graph} and \textit{set of zeros}, respectively. A set-valued operator $M:\scrH \rightrightarrows\scrH$ is called \textit{monotone} if
\begin{equation}
\inner{x-y,x^{\ast}-y^{\ast}}\geq 0\qquad\forall (x^{\ast},x)\in\gr(M),(y^{\ast},y)\in\gr(M).
\end{equation}
The operator $M:\scrH\rightrightarrows\scrH$ is called \emph{maximally monotone} if it is monotone and there is no monotone operator $\tilde{M}:\scrH \rightrightarrows\scrH$ such that $\gr(M)\subseteq\gr(\tilde{M})$. $M$ is said to be \textit{$\rho$-strongly monotone} if
\begin{equation}
\inner{x-y,x^{\ast}-y^{\ast}}\geq \rho\norm{x-y}^{2}\qquad\forall (x^{\ast},x)\in\gr(M),(y^{\ast},y)\in\gr(M).
\end{equation}
If $M$ is maximally monotone and strongly monotone, then $\Zer(M)$ is singleton \cite[][Corollary 23.37]{BC2}.

A single-valued operator $T:\scrH \to \scrH$ is called \emph{nonexpansive} when $\|Tx-Ty\| \leq \|x-y\|$ for all $x, y\in \scrH$, while, given some $\beta > 0$, $T$ is said to be \textit{$\beta$-cocoercive} if $\langle x-y, Tx-Ty\rangle \geq \beta \|Tx-Ty\|^2$ for all $x, y\in \scrH$.

Let $\alpha \in (0,1)$ be fixed. We say that $R : \scrH \to \scrH$ is \textit{$\alpha$-averaged} if there exists a nonexpansive operator $T : \scrH \to \scrH$ such that $R = (1-\alpha) \Id + \alpha T$. An important instance of $\alpha$-averaged operators are \textit{firmly nonexpansive mappings}, which we recover for $\alpha = {1}/{2}$. For further insights into averaged operators we refer the reader to \cite[Section~4.5]{BC2}.

The \textit{resolvent} of the maximally monotone operator $M$ is defined as $J_{M}\eqdef(\Id+M)^{-1}$. It is a single-valued operator with $\dom(J_{M})=\scrH$ and it is firmly nonexpansive
\begin{equation}
\norm{J_{M}x-J_{M}y}^{2}\leq\inner{J_{M}x-J_{M}y,x-y}\qquad\forall x,y\in\scrH.
\end{equation}

For all $\lambda,\mu>0$ and $x\in\scrH$ it holds that (see \cite[Proposition~23.28]{BC2})
\begin{equation}\label{eq:Res}
\norm{J_{\lambda M}x-J_{\mu M}x}\leq\abs{\lambda-\mu}\norm{M_{\lambda}x},
\end{equation}
where $M_{\lambda}\eqdef({1}/{\lambda}) (\Id-J_{\lambda M})$ is the \textit{Yosida approximation} of the maximal monotone operator $M$ with parameter $\lambda>0$.

By $\Pr_{C}$ we denote the \textit{orthogonal projector} onto a closed convex set $C\subseteq \scrH$, while the \textit{normal cone} of a set $C\subseteq \scrH$ is $\NC_C\eqdef\{z\in \scrH :\langle z,y-x\rangle\leq 0$ $\forall y\in C\}$ if $x\in C$ and $\NC_C(x)=\emptyset$ otherwise.

We also need the following basic identity (cf. \cite{BC2})
\begin{align}\label{basic}
	\|\alpha x + (1-\alpha)y \|^2 + \alpha(1 - \alpha) \|x-y\|^2 = \alpha \|x\|^2 + (1-\alpha) \|y\|^2 ~~\forall \alpha \in \mathbb{R}~ \forall x,y \in \scrH.
\end{align}

\subsection{Properties of perturbed operators}

Let $\scrH$ be a real Hilbert space, $A : \scrH \rightrightarrows \scrH$ a maximally monotone operator and $B : \scrH \to \scrH$ a monotone and $({1}/{\beta})$-Lipschitz continuous operator, for some $\beta > 0$.
Let $\varepsilon>0$ and denote $B_{\varepsilon}\eqdef B+\varepsilon\Id:\scrH\to\scrH$. Since $B$ is maximally monotone and $({1}/{\beta})$-Lipschitz, the perturbed operator $B_{\varepsilon}$ is $\eps$-strongly  monotone and $(\varepsilon + {1}/\beta)$-Lipschitz continuous. In particular, the operator $A+B_{\varepsilon}$ is $\eps$-strongly monotone. Hence, for every $\eps>0$ the set $\scrS_{\varepsilon}\eqdef\Zer(A+B_{\varepsilon})$ is a singleton with the unique element denoted $x_{\varepsilon}$. 
Throughout this paper we consider the following hypothesis valid.

\begin{assumption}
$\scrS_{0}\eqdef\Zer(A+B)\neq\emptyset$.
\end{assumption}

The following lemma is a classical result due to Bruck \cite{BRU}. A short proof can be found in \cite[Lemma~4]{CPS}.

\begin{lemma}\label{lem:x_eps}
	It holds $x_\varepsilon \to x^\ast\eqdef\inf\{\norm{x}:x\in\scrS_{0}\}$ as $\varepsilon\to 0$.
\end{lemma}

Lemma \ref{lem:x_eps} implies that the net $(x_{\varepsilon})_{\varepsilon>0}\subset\scrH$ is locally bounded. The next result establishes continuity and differentiability properties of the trajectory $\varepsilon \mapsto x_\eps$. The proof relies on the characterization of zeros of a monotone operator via its resolvent, and can be found in \cite[page 533]{AC6}. For the reader's convenience, we include it here as well. 	

\begin{lemma}\label{diff}
	Let $\varepsilon_1, \varepsilon_2 >0$. Then
	\begin{align*}
	\norm{x_{\varepsilon_1} - x_{\varepsilon_2}} \leq \frac{\norm{x_{\varepsilon_1}}}{\varepsilon_2}\abs{\varepsilon_1 - \varepsilon_2},
	\end{align*}
	i.e. $\varepsilon\mapsto x_\eps$ is locally Lipschitz continuous on $(0,+\infty)$, and therefore differentiable almost everywhere. Furthermore,
	\begin{align*}
	\left\| \frac{\dif }{\dif\varepsilon} x_\varepsilon\right\| \leq \frac{\norm{x_\varepsilon}}{\varepsilon}\ \ \forall \varepsilon \in (0, +\infty).
	\end{align*}
\end{lemma}

\begin{proof}
	First, we observe that, for $\varepsilon>0$, $0 \in Ax_\varepsilon + Bx_\varepsilon + \varepsilon x_\varepsilon$ is equivalent to
$$x_\varepsilon = \left(\operatorname{Id} + \frac{1}{\varepsilon}(A+B)\right)^{-1}(0)=J_{\frac{1}{\varepsilon}(A+B)}(0).$$ Using this fact, and combining it with relation \eqref{eq:Res}, we obtain
	\begin{align*}
\left\| x_{\varepsilon_{1}} - x_{\varepsilon_2}\right\| = \left\| J_{\frac{1}{\varepsilon_1}(A+B)}(0) - J_{\frac{1}{\varepsilon_2}(A+B)}(0)\right\|\leq \left| 1- \frac{\varepsilon_{1}}{\varepsilon_2}\right| \cdot \norm{x_{\varepsilon_1}},
	\end{align*}
	which is equivalent to
	\begin{align*}
\norm{x_{\varepsilon_1} - x_{\varepsilon_2}} \leq \frac{\norm{x_{\varepsilon_1}}}{\varepsilon_2}\abs{\varepsilon_1 - \varepsilon_2}.
	\end{align*}
This proves the first statement.

For the second statement, we note that the previous inequality yields for $\eps=\varepsilon_{1}$ and $\varepsilon_{2}=\eps+h$ the estimate
	\begin{align*}
	0\leq \frac{\norm{x_{\varepsilon} - x_{\varepsilon + h}}}{h} \leq \frac{\norm{x_{\varepsilon}}}{\eps+h}\ \ \forall h\in (0, +\infty).
	\end{align*}
	Passing to the limit $h \to 0$ completes the proof.
\end{proof}

\subsection{Dynamical systems}
In our analysis, we will make use of the following standard terminology from dynamical systems theory.

A continuous function $f:[0,T]\to\scrH$ (where $T>0$) is said to be \textit{absolutely continuous} when its distributional derivative is Lebesgue integrable on $[0,T]$.


We remark that this definition implies that an absolutely continuous function is differentiable almost everywhere, and its derivative coincides with its distributional derivative almost everywhere. Moreover, one can recover the function from its derivative via the integration formula $f(t)=f(0)+\int_{0}^{t}g(s)\dif s$ for all $t\in[0,T]$. 

The solutions of the dynamical systems we are considering in this paper are understood in the following sense.

\begin{definition}
We say that $x:[0,+\infty)\to\scrH$ is a \textit{strong global solution with initial condition $x_{0}\in\scrH$} of the dynamical system
\begin{equation}\label{eq:Dyn}
\left\{\begin{array}{l}
\dot{x}(t) = f(t,x(t))  \\
x(0) = x_0, \end{array}\right.
\end{equation}
where $f : [0,+\infty) \times \scrH \to \scrH$, if the following properties are satisfied:
\begin{itemize}
\item[(a)] $x:[0,+\infty)\to\scrH$ is absolutely continuous on each interval $[0,T],0<T<+\infty$;
\item[(b)] it holds $\dot{x}(t)=f(t,x(t))$  for almost every $t\geq 0$;
\item[(c)] $x(0)=x_{0}$.
\end{itemize}
\end{definition}

Existence and strong uniqueness of nonautonomous systems of the form \eqref{eq:Dyn} can be proven by means of the classical Cauchy-Lipschitz Theorem (see, for instance, \cite[][Proposition 6.2.1]{HAR} or \cite[][Theorem 54]{SON}). To use this, we need to ensure the following properties enjoyed by the vector field $f(\cdot, \dagger)$.
\begin{theorem}\label{th:existence}
Let $f:[0,\infty)\times\scrH\to\scrH$ be a given function satisfying:
\begin{itemize}
\item[(f1)] $f(\cdot,x):[0,+\infty)\to\scrH$ is measurable for each $x\in\scrH$;
\item[(f2)] $f(t,\cdot):\scrH\to\scrH$ is continuous for each $t\geq 0$;
\item[(f3)] there exists a function $\ell(\cdot)\in L^{1}_{\text{loc}}(\R_+;\R)$ such that 
\begin{equation}
\norm{f(t,x)-f(t,y)}\leq\ell(t)\norm{x-y}\qquad\forall t\in [0, b]\ \forall b\in \R_+\ \forall x,y\in\scrH;
\end{equation}
\item[(f4)] for each $x\in\scrH$ there exists a function $\Delta(\cdot)\in L^{1}_{\text{loc}}(\R_+;\R)$ such that 
\begin{equation}
\norm{f(t,x)}\leq\Delta(t)\qquad\forall t\in [0, b]\ \forall b\in \R_+. 
\end{equation}
\end{itemize}
Then, the dynamical system \eqref{eq:Dyn} admits a unique strong solution $t\mapsto x(t)$, $t\geq 0$. 
\end{theorem}


\section{A Tikhonov regularized Krasnoselskii–Mann dynamical system}
\label{sec:km}
Let $T : \scrH \to \scrH$ be a nonexpansive mapping with $\Fix (T) \neq \emptyset$. 
We are interested in investigating the trajectories of the following dynamical system
\begin{align}\label{TKM}
\left\{\begin{array}{ll} \dot{x}(t) = \lambda(t)\left[T(x(t)) - x(t)\right] - \epsilon(t)x(t)  \\
x(0) = x_0. \end{array}\right.
\end{align}
where  $x_0 \in \scrH$ is a given reference point, and $\lambda(\cdot)$ and $\epsilon(\cdot)$ are user-defined functions, satisfying the following standing assumption:
\begin{assumption}\label{ass:lambda}
$\lambda : [0,+\infty) \to (0,1]$ and $\epsilon : [0,+\infty) \to [0,+\infty)$ are Lebesgue measurable functions.
\end{assumption}
Motivated by \cite{BC7}, where it is shown that the trajectories of the dynamical system
$$
\left\{\begin{array}{ll} \dot{x}(t) = \lambda(t)\big(T(x(t)) - x(t)\big)\\
x(0) = x_0, \end{array}\right.
$$
converge weakly towards a fixed point of $T$, and \cite{CPS}, where the strong convergence of the trajectories of a dynamical system involving a maximally monotone operator is induced by means of a Tikhonov regularization, we show that, under mild hypotheses, the trajectories of \eqref{TKM} strongly converge to $\Pr_{\Fix(T)}(0)$, the minimum norm fixed point of $T$. Moreover we also address the question about viability of trajectories in case where $T$ is defined on a nonempty, closed and convex set $D \subseteq \scrH$.

\subsection{Existence and uniqueness of global solutions}
\label{sec:existence}
Existence and uniqueness of solutions to the dynamics \eqref{TKM} follow from the general existence statement, i.e. Theorem \ref{th:existence}. First, notice that the dynamical system \eqref{TKM} can be rewritten in the form of \eqref{eq:Dyn} where $f : [0,+\infty) \times \scrH \to \scrH$ is defined by $f(t,x)\eqdef \lambda(t)[T(x) - x] - \epsilon(t)x$. This shows that properties $(f1)$ and $(f2)$ in Theorem \ref{th:existence} are satisfied. It remains to verify properties $(f3)$ and $(f4)$. 

\begin{lemma}\label{pr1-km}
When $\epsilon \in L^1_{\text{loc}}(\R_+;\R)$, then, for each $x_0 \in \scrH$, there exists a unique strong global solution of \eqref{TKM}.
\end{lemma}

\begin{proof}
$(i)$ Let $x,y \in \scrH$, then, since $T$ is nonexpansive, we have
\begin{align*}
\|f(t,x) - f(t,y)\| \leq 2 \lambda(t) \|x-y\| + \epsilon(t) \|x-y\| = [2 \lambda(t) + \epsilon(t)] \|x-y\|.
\end{align*}
Since $\lambda$ is bounded from above and due to the assumption made on $\epsilon$, one has $\ell(\cdot)\eqdef 2\lambda(\cdot) + \epsilon(\cdot) \in L^1_{\text{loc}}(\R_+;\R)$, so that $(f3)$ holds. 

$(ii)$ For $x \in \scrH$ and $\bar{x}\in\Fix(T)$ one has 
\begin{align*}
\norm{f(t,x)}\leq \norm{f(t,\bar{x})}+\norm{f(t,x)-f(t,\bar{x})}\leq \eps(t)\norm{\bar{x}}+\ell(t)\norm{x-\bar{x}}\equiv\Delta(t) 
\end{align*}
for any $t\in [0, +\infty)$. Existence and uniqueness now follow from Theorem \ref{th:existence}.
\end{proof}

\subsection{Convergence of the trajectories: first approach}

The following observation, which is based on a time rescaling argument similar to \cite[Lemma~4.1]{AC2}, will be fundamental for the convergence analysis of the trajectories. We give it without proof since it can be derived as a special case of Theorem \ref{rescalD}.

\begin{theorem}\label{equiv}
Let $\tau_1 : [0,+\infty) \to [0,+\infty)$ be the function which is implicitly defined by 
\[ 
\int_{0}^{\tau_1(t)} \lambda(s) \dif s = t, \quad\tau_{1}(0)=0.
\]
Similarly, let $\tau_2 : [0,+\infty) \to [0,+\infty)$ be the function given by 
\begin{align*}
 \tau_2(t) \eqdef \int_{0}^{t} \lambda(s) \dif s. 
 \end{align*}
 Set $\tilde{\epsilon} \eqdef \epsilon \circ \tau_1$, $\tilde{\lambda} \eqdef \lambda \circ \tau_1$, and consider the system
\begin{align}\label{TKM1}
\left\{\begin{array}{ll} 
\dot{u}(t) = T(u(t)) -  u(t) - \frac{\tilde{\epsilon}(t)}{\tilde{\lambda}(t)} u(t)  \\
u(0) = x_0, \end{array}\right.
\end{align}
where $x_0\in \scrH$. If $x$ is a strong solution of \eqref{TKM}, then $u\eqdef x \circ \tau_1$ is a strong solution of the system \eqref{TKM1}. Conversely, if $u$ is a strong solution of \eqref{TKM1}, then $x\eqdef u \circ \tau_2$ is a strong solution of the system \eqref{TKM}.
\end{theorem}

Theorem \ref{equiv} suggests that one can also study the dynamical system \eqref{TKM1} instead of \eqref{TKM}. Moreover, in \cite[Theorem 9]{CPS} the strong convergence of the trajectories of the differential inclusion
\begin{align}\label{A}
\left\{\begin{array}{ll} - \dot{u}(t) \in A(u(t)) + \epsilon(t)u(t)  \\
u(0) = x_0, \end{array}\right.
\end{align}
where $A : \scrH \rightrightarrows \scrH$ is a maximally monotone operator such that $A^{-1}(0)\neq\emptyset$, towards the minimum norm zero of $A$ was obtained provided that 
$\lim_{t\to\infty}\epsilon(t) = 0$, $\epsilon \notin L^{1}(\R_+;\R)$ and $\abs{\dot{\epsilon}}\in  L^{1}(\R_+;\R)$. 
The connection between \eqref{A} and \eqref{TKM} (as well as \eqref{TKM1}) is achieved through the fact that the nonexpansiveness of $T$ guarantees that the operator $A \eqdef \Id - T$ is maximally monotone and, furthermore, $x \in A^{-1}(0)$ holds if and only if $x \in \Fix T$.

Now we establish the convergence of the trajectories of the dynamical system \eqref{TKM}, noting that the employed hypotheses coincide with those of \cite[Theorem 9]{CPS} when $\lambda (t) =1$ for all $t\in [0, +\infty)$.

\begin{theorem}\label{main2}
	Let $t\mapsto x(t)$, $t\geq 0$, be the strong solution of \eqref{TKM} with initial condition $x^{0}\in\scrH$, and assume that
	\begin{align*}
	(i) ~&\int_{0}^{+\infty} \epsilon(t) dt = +\infty, \\[5pt]
	(ii) ~&\int_{0}^{+\infty} \lambda(t) dt = +\infty,  \\[5pt]
	(iii) ~& \epsilon \mbox { and } \lambda \mbox{ are absolutely continuous and } \frac{\epsilon(t)}{\lambda(t)} \to 0 \text{ as } t \to +\infty, \\[5pt]
	(iv) ~&\int_{0}^{+\infty} \left| \frac{d}{dt} \left(\frac{\epsilon(t)}{\lambda(t)}\right) \right| dt < +\infty.
	\end{align*}
	Then $x(t) \to \Pr_{\Fix(T)}(0)$ as $t \to +\infty$.
\end{theorem}

\begin{proof}
By Theorem \ref{equiv}, the dynamical system \eqref{TKM1} has a strong solution, too. Since $\Id - T$ is maximally monotone, and $x \in \Fix T$ holds if and only if $x \in (\Id - T)^{-1}(0)$, we verify first that the function $\frac{\tilde{\epsilon}}{\tilde{\lambda}}=\frac{\eps\circ\tau_{1}}{\lambda\circ\tau_{1}}$ fulfills the assumptions of \cite[Theorem 9]{CPS}.
First,
\[ \int_{0}^{+\infty} \frac{\tilde{\epsilon}(t)}{\tilde{\lambda}(t)} dt = \int_{0}^{+\infty} \dot{\tau_1}(t) \epsilon(\tau_1(t)) dt = \int_{0}^{+\infty} \epsilon(s) ds = +\infty. \]
Further, for almost all $t\geq 0$ it holds, taking into consideration that $\dot{\tau_1}(t) \lambda (t) = 1$,
\begin{align*}
\frac{d}{dt} \left(\frac{\tilde{\epsilon}(t)}{\tilde{\lambda}(t)}\right)
&= \frac{\dot{\tilde{\epsilon}}(t) \tilde{\lambda}(t) - \tilde{\epsilon}(t) \dot{\tilde{\lambda}}(t)}{\tilde{\lambda}(t)^2} \\
&= \frac{\dot{\tau_1}(t)\dot{\epsilon}(\tau_1(t)) \tilde{\lambda}(t) - \tilde{\epsilon}(t) \dot{\tau_1}(t) \dot{\lambda}(\tau_1(t))}{\tilde{\lambda}(t)^2} \\
&= \frac{\dot{\epsilon}(\tau_1(t))}{\lambda(\tau_1(t))^2} - \frac{\epsilon(\tau_1(t)) \dot{\lambda}(\tau_1(t))}{\lambda(\tau_1(t))^3},
\end{align*}
hence
\begin{align*}
\int_{0}^{+\infty} \left| \frac{d}{dt} \left(\frac{\tilde{\epsilon}(t)}{\tilde{\lambda}(t)}\right) \right| dt &= \int_{0}^{+\infty} \left| \frac{\dot{\epsilon}(\tau_1(t))}{\lambda(\tau_1(t))^2} - \frac{\epsilon(\tau_1(t)) \dot{\lambda}(\tau_1(t))}{\lambda(\tau_1(t))^3} \right| dt \nonumber \\
&= \int_{0}^{+\infty} \left| \frac{\dot{\epsilon}(s)}{\lambda(s)} - \frac{\epsilon(s) \dot{\lambda}(s)}{\lambda(s)^2} \right| ds = \int_{0}^{+\infty} \left| \frac{d}{dt} \left(\frac{\epsilon(t)}{\lambda(t)}\right) \right| dt < +\infty,
\end{align*}
where we used that $\tau_1(t) \to +\infty$ as $t \to +\infty$. Therefore the strong convergence of the strong solution of \eqref{TKM1} with initial condition $x^{0}\in\scrH$ towards $\Pr_{\Fix(T)}(0)$ is proven. The assertion follows by Theorem \ref{equiv}.
\end{proof}

\subsection{Convergence of the trajectories: second approach}

Our second convergence statement concerns a generalization of the system \eqref{TKM} where $T$ maps from a closed and convex set $D\subseteq\scrH$ to $D$. For such dynamics, a key condition is to ensure invariance with respect to the domain $D$ of the trajectory $t\mapsto x(t)$, $t\geq 0$, when issued from an initial condition $x_{0}\in D$. Such viability results are key to the control of dynamical systems \cite{Aub91,AubCel84}. To that end, we consider the differential equation
\begin{align}\label{TKMDl}
\left\{\begin{array}{ll} \dot{x}(t) = \lambda(t)\big(T(x(t)) - x(t)\big) - \epsilon(t)(x(t)-y)  \\
x(0) = x_0 \in D, \end{array}\right.
\end{align}
where $y \in D$ is fixed reference point and $\Fix (T)\neq\emptyset$. In the very recent note \cite{PAV} the strong convergence of the trajectory for the case $\lambda(t) = 1$ for all $t\in [0,+\infty)$ towards $\Pr_{\Fix(T)}(y)$ has been demonstrated in \cite[Theorem 4.1]{PAV} by assuming that $\epsilon \in L^1_{\text{loc}}(\R_+; \R)$ is absolutely continuous and nonincreasing, $\epsilon(t) \to 0$ as $t \to +\infty$, $\int_{0}^{+\infty} \epsilon(s) ds = +\infty$ and $\lim_{t \to +\infty} {\dot{\epsilon}(t)}/{\epsilon^2(t)} = 0$. 

First we give the existence and uniqueness statement for the strong global solution of \eqref{TKMDl}, whose proof is skipped as it follows Proposition \ref{pr1-km} and \cite[Proposition 4.1]{PAV}. 

\begin{proposition}\label{ExistenceUniquenessD}
Assume that $\epsilon \in L^1_{\text{loc}}(\R_+; \R)$. Then, for any pair $(x_{0},y)\in D\times D$ the dynamical system \eqref{TKMDl} admits a unique strong global solution $t\mapsto x(t)$, $t\geq 0$ which leaves the domain $D$ forward invariant, i.e. $x(t) \in D$ for all $t \in [0,+\infty)$.
\end{proposition}


A result similar to Theorem \ref{equiv} for \eqref{TKMDl} is provided next.

\begin{theorem}\label{rescalD}
Let $\tau_1 : [0,+\infty) \to [0,+\infty)$ be the function implicitly defined by 
\begin{align*}
 \int_{0}^{\tau_1(t)} \lambda(s) ds = t, \quad\tau_{1}(0)=0.
 \end{align*}
Furthermore, let $\tau_2 : [0,+\infty) \to [0,+\infty)$ the function given by 
\[ 
\tau_2(t) \eqdef \int_{0}^{t} \lambda(s) ds. 
\] 
Set $\tilde{\epsilon} \eqdef \epsilon \circ \tau_1$, $\tilde{\lambda} \eqdef \lambda \circ \tau_1$ and consider the system
\begin{align}\label{TKMlD}
\left\{\begin{array}{ll} \dot{u}(t) = T(u(t)) -  u(t) - \frac{\tilde{\epsilon}(t)}{\tilde{\lambda}(t)} (u(t) - y)  \\
u(0) = x_0, \end{array}\right.
\end{align}
where $(x_0, y)\in D\times D$. If $t\mapsto x(t)$, $t\geq 0$, is the strong solution of \eqref{TKMDl}, then $u\eqdef x \circ \tau_1$ is a strong solution of the system \eqref{TKMlD}. Conversely, if $u$ is a strong solution of \eqref{TKMlD}, then $x\eqdef u\circ \tau_2$ is a strong solution of the system \eqref{TKMDl}.
\end{theorem}

\begin{proof}
Let $t\mapsto x(t)$, $t\geq 0$, be a strong solution of \eqref{TKMDl}. Since we already know that $x(t) \in D$ for all $t \geq 0$, the first line of \eqref{TKMDl} written at point $T_1(t)$ and multiplied by $\dot{\tau_1}(t)$ yields
\[\dot{\tau_1}(t) \dot{x}(\tau_1(t)) = \dot{\tau_1}(t) \lambda(\tau_1(t))[T(x(\tau_1(t))) - x(\tau_1(t))] - \dot{\tau_1}(t) \epsilon(\tau_1(t))[x(\tau_1(t)) - y]. \]
Since $u(t)\eqdef x(\tau_1(t))$, $\dot{u}(t) = \dot{\tau_1}(t) \dot{x}(\tau_1(t))$ and $\dot{\tau_1}(t) = 1/\lambda(\tau_1(t))$, we obtain from the line above \[\dot{u}(t) = T(u(t)) -  u(t) - \frac{\tilde{\epsilon}(t)}{\tilde{\lambda}(t)} (u(t) - y)\]
for almost every $t\geq 0$. Moreover, $u(0)=x(\tau_1(0))=x_0$.

Now, let $u$ be a strong solution of $\eqref{TKM1}$. From \cite[Proposition~4.1]{PAV} we deduce that that $u(t) \in D$ for all $t \geq 0$. The first line of \eqref{TKMlD} written at point $\tau_2(t)$ and multiplied by $\dot{\tau_2}(t)$ reads
\[ 
\dot{\tau_2}(t) \dot{u}(\tau_2(t)) = \dot{\tau_2}(t) \big(T(u(\tau_2(t))) -  u(\tau_2(t))\big) - \dot{\tau_2}(t) \frac{\tilde{\epsilon}(\tau_2(t))}{\tilde{\lambda}(\tau_2(t))} [u(\tau_2(t)) - y]\ \forall t\geq 0. 
\]
Observing that $x(t) = u(\tau_2(t))$, $\dot{x}(t) = \dot{\tau_2}(t) \dot{u}(\tau_2(t))$, $\dot{\tau_2}(t) = \lambda(t)$ and $\tau_1 \circ \tau_2 = \Id$, the previous line becomes for almost every $t\geq 0$
\[
\dot{x}(t) = \lambda(t)\big(T(x(t)) - x(t)\big) - \epsilon(t)(x(t) - y).
\]
Moreover, $x(0)=u(0)=x_0$. This concludes the proof.
\end{proof}

Employing the time rescaling arguments from Theorem \ref{rescalD}, we are able to derive the following statement, which extends \cite[Theorem 4.1]{PAV} that is recovered as special case when $\lambda(t) = 1$ for all $t\in [0,+\infty)$.

\begin{theorem}\label{main1}
	Let $t\mapsto x(t)$, $t\geq 0$, be the strong solution of \eqref{TKMDl} and assume that
	\begin{enumerate}
	\item[(i)] $\int_{0}^{+\infty} \epsilon(t) dt = +\infty$, 
	\item[(ii)] $\int_{0}^{+\infty} \lambda(t) dt = +\infty$,
	\item[(iii)] $\epsilon$  and $\lambda$  are absolutely continuous, $\frac{\epsilon(\cdot)}{\lambda(\cdot)}$  is nonincreasing and $\frac{\epsilon(t)}{\lambda(t)} \to 0$  as $t \to +\infty$,
	\item[(iv)] $\frac{\dot{\epsilon}(t)}{\epsilon(t)^2} - \frac{\dot{\lambda}(t)}{\lambda(t) \epsilon(t)} \to 0$  as $t \to +\infty$.
	\end{enumerate}
		Then $x(t) \to \Pr_{\Fix(T)}(y)$ as $t \to +\infty$.
\end{theorem}

\begin{proof}
	In a similar manner to the proof of Theorem \ref{main2}, due to Theorem \ref{rescalD} it suffices to check the assumptions in \cite[Theorem 4.1]{PAV} for the function ${\tilde{\epsilon}}/{\tilde{\lambda}}$. First, we notice that
	\[ \int_{0}^{+\infty} \frac{\tilde{\epsilon}(t)}{\tilde{\lambda}(t)} dt = \int_{0}^{+\infty} \dot{\tau_1}(t) \epsilon(\tau_1(t)) dt = \int_{0}^{+\infty} \epsilon(s) ds = +\infty, \]
	where we used that $\tau_1(t) \to +\infty$ as $t \to +\infty$. From the proof of Theorem \ref{main2} we know that for almost all $t\geq 0$ one has
	\begin{align*}
	\frac{d}{dt} \frac{\tilde{\epsilon}(t)}{\tilde{\lambda}(t)} = \frac{\dot{\epsilon}(\tau_1(t))}{\lambda(\tau_1(t))^2} - \frac{\epsilon(\tau_1(t)) \dot{\lambda}(\tau_1(t))}{\lambda(\tau_1(t))^3},
	\end{align*}
	The last expression divided by $\left(\frac{\tilde{\epsilon}(t)}{\tilde{\lambda}(t)}\right)^2$ gives
	\[ 
	\frac{\dot{\epsilon}(\tau_1(t))}{\epsilon(\tau_1(t))^2} - \frac{\dot{\lambda}(\tau_1(t))}{\lambda(\tau_1(t)) \epsilon(\tau_1(t))}, 
	\]
	which, due to the assumptions we made on the functions $\epsilon$ and $\lambda$, tends to $0$ as $t \to +\infty$. 
\end{proof}

In the following two remarks we compare the hypotheses of Theorem \ref{main2} and Theorem \ref{main1}, noting that, despite the common assumptions $(i)-(ii)$, they do not fully cover each other.

\begin{remark}
The framework of Theorem \ref{main1} extends the one of Theorem \ref{main2} by allowing the involved operator $T$ to map a closed convex set to itself, the latter being recovered when choosing $D = \scrH$ and $y = 0$. However, in this setting,
fixing $\beta \in (0,1)$ and taking $\epsilon(t) = {1}/{(0.2+t)^\beta}$ and $\lambda(t) = 0.5 \cos(\frac{1}{0.2+t}) + 0.5$, $t\geq 0$, one notes that $\lambda(t) \in [0,1]$ $\forall t\geq 0$, $\int_{0}^{+\infty} \epsilon(t) dt = \int_{0}^{+\infty} \lambda(t) dt = +\infty$ and ${\epsilon(t)}/{\lambda(t)}$ is converging to $0$ as $t \to +\infty$, but there exists an intervall where the function is increasing. Hence assumption (iii) of Theorem \ref{main1} is violated while the corresponding assumption in Theorem \ref{main2} is fulfilled. Moreover,
$$\frac{d}{dt} \left( \frac{\epsilon(t)}{\lambda(t)} \right) = \frac{-\beta (0.1 + t)^{-(\beta + 1)}}{(0.5\cos(\frac{1}{0.2 + t})+0.5)} - \frac{0.5\sin(\frac{1}{0.2 + t})}{(0.5\cos(\frac{1}{0.2 + t})+0.5)^2 (0.2+t)^{\beta + 2}}\ \forall t\geq 0,$$
that is a function of class $L^1(\R_+; \R)$. Hence, for the chosen parameter functions $\epsilon$ and $\lambda$, the assumptions of Theorem \ref{main1} are not satisfied, while the ones of Theorem \ref{main2} are.
\begin{figure}[htbp]
	\centering
	\includegraphics[width=1\textwidth]{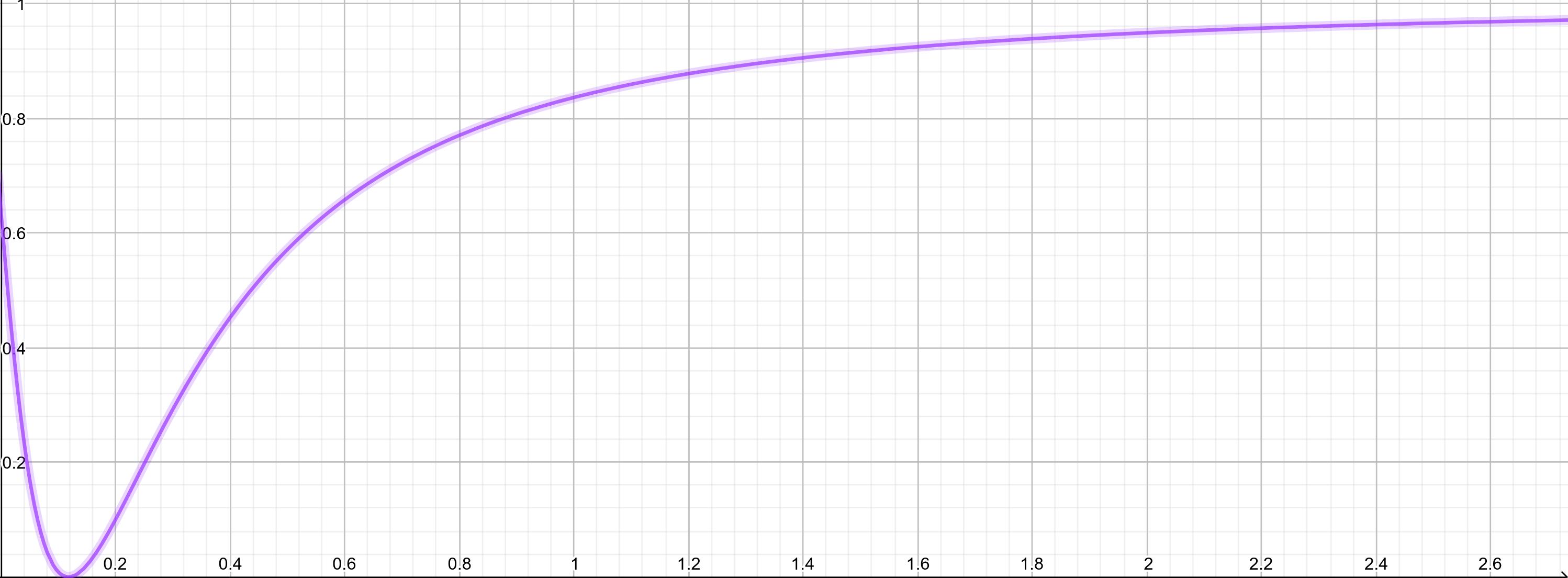}
	\caption{Graph of $\lambda(t) = 0.5 \cos(\frac{1}{0.2+t}) + 0.5$}
\end{figure}
\end{remark}

\begin{remark}
In the situation of Theorem \ref{main1} we consider again the choice $D = \scrH$ and $y = 0$. In this case Theorem \ref{main2} is a special instance of Theorem \ref{main1}. In fact, since $\epsilon/{\lambda}$ is assumed to be nonincreasing and ${\epsilon(t)}/{\lambda(t)} \to 0$ as $t \to +\infty$ we conclude that
$$\int_{0}^{+\infty} \left| \frac{d}{dt} \left(\frac{\epsilon(t)}{\lambda(t)}\right) \right| dt = - \int_{0}^{+\infty} \frac{d}{dt} \left(\frac{\epsilon(t)}{\lambda(t)}\right) dt = - \lim\limits_{s \to +\infty} \frac{\epsilon(s)}{\lambda(s)} + \frac{\epsilon(0)}{\lambda(0)} < +\infty,$$
i.e. assertion (iv) of Theorem \ref{main2} is fulfilled.
\end{remark}

\begin{remark}
One can also compare the hypotheses imposed in Theorem \ref{main2} and Theorem \ref{main1} for guaranteeing the strong convergence of the trajectories of a dynamical system towards a fixed point of $T$ with the ones required in \cite[Theorem 6 and Remark 17]{BC7}, the weakest of them being $(ii)$ of any of Theorem \ref{main2} and Theorem \ref{main1}. Taking also into consideration \cite[Proposition~5 and Theorem 9]{CPS} as well as \cite[Proposition 4.1 and Theorem 4.1]{PAV}, the assumptions of both
Theorem \ref{main2} and Theorem \ref{main1} turn out to be natural for achieving the strong convergence of the trajectories of the dynamical system \eqref{TKM} towards a fixed point of $T$.
\end{remark}

\subsection{Special case: an outer Tikhonov regularized forward-backward dynamical system}

From the analysis of the strong convergence of the trajectories of the Tikhonov regularized Krasnoselskii-Mann dynamical system \eqref{TKM} one can deduce similar assertions for determining zeros of a sum of monotone operators. Let $A : \scrH \rightrightarrows \scrH$ be a maximally monotone operator and $B : \scrH \to \scrH$ a $\beta$-cocoercive operator with $\beta > 0$ such that $\Zer(A+B)$ is nonempty. The dynamical system employed to this end is a Tikhonov regularized version of \cite[equation (14)]{BC7}, namely, when  $\gamma \in (0,2\beta)$, $\epsilon : [0,+\infty) \to [0,+\infty)$ and $\lambda : [0,+\infty) \to [0,({4\beta - \gamma})/({2\beta})]$ are Lebesgue measurable functions, and $x_0 \in \scrH$,
\begin{align}\label{FBx}
	\left\{\begin{array}{ll} \dot{x}(t) = \lambda(t)\big(J_{\gamma A}\big(x(t) - \gamma B(x(t)) \big) - x(t)\big) - \epsilon(t)x(t)  \\
	x(0) = x_0. \end{array}\right.
	\end{align}

Employing either Theorem \ref{main2} or Theorem \ref{main1}, one can derive the following statement.
	
\begin{theorem}\label{FBout}
	 Suppose that either the assumptions of Theorem \ref{main2} or Theorem \ref{main1} made on the parameter functions $\epsilon$ and $\lambda$ are fulfilled. Further, let $x$ be the unique strong global solution of the dynamical system \eqref{FBx}. Then $x(t) \to \Pr_{\Zer(A+B)}(0)$ as $t \to +\infty$.
\end{theorem}

\begin{proof}
	Since the resolvent of a maximally monotone operator is firmly nonexpansive it is $1/2$-averaged, see \cite[Remark 4.34(iii)]{BC2}. Moreover, by \cite[Proposition~4.39]{BC2} is ${\gamma}/(2\beta)$-averaged. Combining these two observations with \cite[Proposition 4.44]{BC2} yields that the composed operator $T \eqdef J_{\gamma A} \circ (\Id - \gamma B)$ is ${2\beta}/({4\beta - \gamma})$-averaged. Further, it is immediate that the dynamical system \eqref{FBx} can be equivalently written as
	\begin{align*}
	\left\{\begin{array}{ll} \dot{x}(t) = \lambda(t)(T(x(t)) - x(t)) - \epsilon(t)x(t)  \\
	x(0) = x_0. \end{array}\right.
	\end{align*}
	As $T$ is ${2\beta}/({4\beta - \gamma})$-averaged, there exists a nonexpansive operator $\hat T: \scrH \to \scrH$ such that $T= (1- {2\beta}/({4\beta - \gamma})) \Id + ({2\beta}/({4\beta - \gamma})) \hat T$.	
	Then the dynamical system \eqref{FBx} can be further equivalently written as
	\begin{align*}
	\left\{\begin{array}{ll} \dot{x}(t) = \lambda(t)\frac{2\beta}{4\beta - \gamma}(\hat T(x(t)) - x(t)) - \epsilon(t)x(t)  \\
	x(0) = x_0. \end{array}\right.
	\end{align*}
	
	Since $\Fix \hat T = \Fix T =\Zer(A+B)$ (see \cite[Proposition 26.1(iv)(a)]{BC2}) the assertion follows from Theorem \ref{main2} or Theorem \ref{main1}.
\end{proof}

\begin{remark}
	The strong convergence of the trajectories of a forward-backward dynamical system was achieved in \cite[Theorem 12]{BC7} unde the more demanding hypothesis of uniform monotonicity
	(recall that an operator $T: \scrH \to \scrH$ is said to be \textit{uniformly monotone} if there exists an increasing function $\Phi_T : [0,+\infty) \to [0,+\infty]$ that vanishes only at $0$ such that $\langle x-y, u-v\rangle \geq \Phi_T(\|x - y\|)$ for every $(x, u), (y, v) \in \gr (T)$) imposed on one of the involved operators.
\end{remark}

\section{A Tikhonov regularized forward-backward dynamical system}
\label{sec:fb}
In this section we construct Tikhonov regularized dynamical systems which are strongly converging to solutions of \eqref{eq:MIP}. The problem formulation consists a maximally monotone operator $A : \scrH \rightrightarrows \scrH$ and $B : \scrH \to \scrH$ a $\beta$-cocoercive operator with $\beta > 0$ such that $\Zer(A+B)$ is nonempty. Moreover, for $t\in [0, +\infty)$ denote $B_{\eps (t)}\eqdef B+\eps(t)\Id:\scrH\to\scrH$ and $\Zer(A+B_{\eps(t)})=\{\bar{x}(\eps(t))\}$. We consider the dynamical system
\begin{align}\label{FB}
\left\{\begin{array}{ll} \dot{x}(t) &= \lambda(t) \left(J_{\gamma(t) A} \bigg(x(t) - \gamma(t) (Bx(t) + \epsilon(t)x(t))\bigg) - x(t) \right)  \\
x(0) &= x_0, \end{array}\right.
\end{align}
where $\lambda(\cdot),\eps(\cdot)$ obey Assumption \ref{ass:lambda}, and $\gamma : [0, + \infty) \to (0,2\beta)$.

\begin{remark}
Comparing \eqref{FB} with \eqref{FBx} 
 one can note two differences: First of all, in \eqref{FB} the stepsizes are provided by the function $\gamma : [0, + \infty) \to (0,2\beta)$, while in \eqref{FBx} $\gamma$ is a positive constant lying in the interval $(0,2\beta)$ as well. Secondly, to get \eqref{FBx} from the forward-backward dynamical system (cf. \cite{BC7})
\begin{align}\label{FBo}
\left\{\begin{array}{ll} \dot{x}(t) &= \lambda(t) \big(J_{\gamma A} \left(x(t) - \gamma Bx(t)\right) - x(t) \big)  \\
x(0) &= x_0, \end{array}\right.
\end{align}
an outer perturbation is employed, while for \eqref{FB} an inner one. As illustrated in Section \ref{sec:illustrations}, this leads to different performances in concrete applications.
\end{remark}

\subsection{Existence and uniqueness of strong global solutions}

The dynamical system $\eqref{FB}$ can be rewritten as
\begin{align*}
\left\{\begin{array}{ll} \dot{x}(t) = f(t,x(t))  \\
x(0) = x_0, \end{array}\right.
\end{align*}
where $f : [0,+ \infty) \times \scrH \to \scrH$ is defined by $f(t,x) = \lambda(t)(T_t (x) - x)$, with $T_t \eqdef J_{\gamma(t) A } (\Id - \gamma(t) B_{\epsilon(t)})$. Hence, existence and uniqueness of trajectories follows by verifying the conditions spelled out in Theorem \ref{th:existence}. 

\begin{proposition}\label{pr1-fb}
	Assume that $\epsilon$ is of class $L^1_{\text{loc}}(\R_+;\R)$. Then, for each $x_0 \in \scrH$, there exists a unique strong global solution $t\mapsto x(t)$, $t\geq 0$, of \eqref{FB}.
\end{proposition}

\begin{proof}
Conditions $(f1),(f2)$ are clearly satisfied. To show $(f3)$, let $x,y \in \scrH$ be arbitrary. Since $B$ is $(1/\beta)$-Lipschitz continuous, the perturbed operator $B_{\epsilon(t)}$ is $(({1}/{\beta}) + \epsilon(t))$-Lipschitz continuous as well. Hence, by nonexpansivity of the resolvent, we obtain for all $t \in [0,+\infty)$
	\begin{align*}
	\|f(t,x) - f(t,y)\| &\leq \lambda(t) \|x-y\| + \lambda(t) \|T_tx - T_t y\| \\
	&\leq \lambda(t) \|x-y\| + \lambda(t) \|(x - y) - \gamma(t)(B_{\epsilon(t)}x - B_{\epsilon(t)}y)\| \\
	&\leq \lambda(t)\left(2 + \gamma(t) \left[\frac{1}{\beta} + \epsilon(t)\right] \right) \|x-y\|
	\end{align*}
	Since $\lambda$ and $\gamma$ are bounded and due to the assumption we imposed on $\epsilon$, one has $\ell(\cdot)\eqdef \lambda(\cdot)\left(2 + \gamma(\cdot) [{1}/{\beta} + \epsilon(\cdot)] \right) \in L^1_{\text{loc}}(\R_+;\R)$. Condition $(f4)$ is verified by first noting that for $\bar{x} \in \Zer(A+B)$, we have $\bar{x} = T_{t}(\bar{x})$ for all $t\geq 0$. 
	Therefore, for all $x \in \scrH$ and all $t\geq 0$ it holds
	\begin{align*}
	\|T_t(x) - \bar{x}\| &= \left\|J_{\gamma(t) A} (x - \gamma(t) B_{\epsilon(t)} x) - J_{\gamma(t) A}(\bar{x} - \gamma(t) B\bar{x})\right\| \nonumber \\
	&\leq \|(x - \bar{x}) - \gamma(t)(B_{\epsilon(t)} x - B\bar{x}) \| \\
	&\leq \|x - \bar{x}\| + \frac{\gamma(t)}{\beta} \|x - \bar{x}\| + \gamma(t) \epsilon(t)\norm{x}\\
	\end{align*}
	Hence, for all $x \in \scrH$ and all $t\geq 0$ one has
	\begin{align*}
	\norm{f(t,x)}&=\lambda(t)\norm{T_{t}(x)-x}\\
	&\leq \lambda(t)\norm{T_{t}(x)-\bar{x}}+\lambda(t)\norm{x-\bar{x}}\\
	&\leq\lambda(t)\left(2+\frac{\gamma(t)}{\beta}\right) \|x - \bar{x}\|+ \lambda(t)\gamma(t) \epsilon(t)\norm{x}
	\equiv\Delta(t).
	\end{align*}
	Therefore, $(f4)$ holds as well.
\end{proof}

\subsection{Convergence of the trajectory}

As a preliminary step for proving the convergence statement of the trajectories of \eqref{FB} towards $P_{\Zer(A+B)}(0)$ we need the following auxiliary result. Recall that $\Zer(A+B_{\eps(t)})= \Fix(T_t) = \{\bar{x}(\eps(t))\}$.

\begin{lemma}\label{le1-fb}
		Let $t\mapsto x(t)$, $t\geq 0$, be the strong global solution of \eqref{FB} and suppose that $\gamma(t) \leq {2 \beta}/{1+ 2\beta \epsilon(t)}$ for all $t \in [0,+\infty)$. Then, for almost all $t \in [0,+\infty)$
		\begin{align*}
		\langle \dot{x}(t), x(t) - \bar{x}(\epsilon(t)) \rangle \leq \frac{\lambda(t)}{2} \gamma(t) \epsilon(t) (\gamma(t) \epsilon(t) - 2) \|x(t) - \bar{x}(\epsilon(t))\|^2.
		\end{align*}
\end{lemma}

\begin{proof}
By $\eqref{basic}$ we get for almost all $t \in [0,+\infty)$
\begin{align}\label{binomial}
	2 \langle \dot{x}(t), x(t) - \bar{x}(\epsilon(t)) \rangle =& \|\dot{x}(t) + x(t) - \bar{x}(\epsilon(t))\|^2 - \|\dot{x}(t)\|^2 - \|x(t) - \bar{x}(\epsilon(t))\|^2 \nonumber \\
	=& \|\lambda(t) (T_t(x(t)) - \bar{x}(\epsilon(t))) + (1 - \lambda(t)) (x(t) - \bar{x}(\epsilon(t))) \|^2  \nonumber \\
	-&  \|\dot{x}(t)\|^2 - \|x(t) -\bar{x}(\epsilon(t))\|^2 \nonumber \\
	=& \lambda(t) \|T_t(x(t)) - \bar{x}(\epsilon(t))\|^2 + (1 - \lambda(t)) \|x(t) - \bar{x}(\epsilon(t))\|^2 \nonumber  \\
	-& \lambda(t) (1 - \lambda(t)) \|T_t(x(t)) - x(t)\|^2 - \|\dot{x}(t)\|^2 - \|x(t) \nonumber \\
	- & \bar{x}(\epsilon(t))\|^2 \nonumber \\
	=& \lambda(t) \|T_t(x(t)) -\bar{x}(\epsilon(t))\|^2 - \lambda(t) \|x(t) -\bar{x}(\epsilon(t))\|^2 \nonumber \\
	- & \lambda(t) (1 - \lambda(t)) \|T_t(x(t)) - x(t)\|^2 - \|\dot{x}(t)\|^2.
\end{align}
On the other hand, for $x,y \in \scrH$ and for all $t \in [0,+\infty)$ we obtain
\begin{align}\label{argument}
	\|(\Id - \gamma(t) B_{\epsilon(t)})x - & (\Id - \gamma(t) B_{\epsilon(t)})y\|^2 = \|(1 - \gamma(t) \epsilon(t))(x-y) - \gamma(t) (Bx - By)\|^2 \nonumber \\
	 = & (1-\gamma(t) \epsilon(t))^2 \|x-y\|^2 + \gamma(t)^2 \|Bx - By\|^2 \nonumber  \\
	 - & 2 \gamma(t) (1 - \gamma(t) \epsilon(t)) \langle x - y, Bx - By \rangle \nonumber \\
 \leq & (1-\gamma(t) \epsilon(t))^2 \|x-y\|^2 \nonumber \\ 
 + & [\gamma(t)^2 - 2\gamma(t)\beta(1 - \gamma(t) \epsilon(t))] \|Bx - By\|^2,
	\end{align}
where we used the $\beta$-cocoercivity of $B$ in the last step and the observation that $\gamma(t) \epsilon(t)\leq 1$ due to the hypothesis.

By assumption, $\gamma(t) \leq 2 \beta (1-\gamma(t) \epsilon(t))$ for all $t \in [0,+\infty)$. Therefore relation \eqref{argument} yields
\begin{align*}
	\|(\Id - \gamma(t) B_{\epsilon(t)})x - (\Id - \gamma(t) B_{\epsilon(t)})y\|^2 \leq (1-\gamma(t) \epsilon(t))^2 \|x-y\|^2,
\end{align*}
and by the nonexpansivity of the resolvent
\begin{align}\label{Lipschitz}
	\|T_t x - T_t y\|^2 \leq (1-\gamma(t) \epsilon(t))^2 \|x-y\|^2\ \forall t \in [0,+\infty).
\end{align}

Combining \eqref{binomial} with \eqref{Lipschitz} by neglecting the two nonpositive terms in the last line of $\eqref{binomial}$ yields for almost all $t \in [0,+\infty)$
\begin{align*}
	2 \langle \dot{x}(t), x(t) - \bar{x}(\epsilon(t)) \rangle &\leq \lambda(t) (1-\gamma(t) \epsilon(t))^2 \|x(t) - \bar{x}(\epsilon(t))\|^2 - \lambda(t) \|x(t) - \bar{x}(\epsilon(t))\|^2 \\
	&= \lambda(t) \gamma(t) \epsilon(t) (\gamma(t) \epsilon(t) - 2) \|x(t) - \bar{x}(\epsilon(t))\|^2.
\end{align*}
This completes the proof.
\end{proof}

The convergence statement follows.

\begin{theorem}\label{th1-fb}
	Let $t\mapsto x(t)$, $t\geq 0$, be the strong solution of \eqref{FB}. Suppose that $\gamma(t) \leq \frac {2 \beta} {1+ 2\beta \epsilon(t)}$ for all $t \in [0,+\infty)$ and that the following properties are fulfilled
	\begin{align*}
	(i) ~& \epsilon \text{ is absolutely continuous and} \epsilon(t)\text{ decreases to } 0 \text{ as } t \to + \infty, \\[5pt]
	(ii) ~&\frac{\dot{\epsilon}(t)}{\epsilon^2(t)\lambda (t)\gamma (t)} \to 0 \text{ as } t \to + \infty, \\[5pt]
	(iii) ~&\int_{0}^{+\infty} \lambda(t) \gamma(t) \epsilon(t) (2 - \gamma(t) \epsilon(t)) dt = + \infty.
	\end{align*}
	Then $x(t) \to P_{\Zer(A+B)}(0)$ as $t \to + \infty$.
\end{theorem}

\begin{proof}
	Set $\theta(t) \eqdef \frac{1}{2} \|x(t) - \bar{x}(\epsilon(t))\|^2$, $t\geq 0$. Then, by using Lemma \ref{le1-fb}
	\begin{align*}
	\dot{\theta}(t) &= \left\langle x(t) - \bar{x}(\epsilon(t)), \dot{x}(t) - \dot{\epsilon}(t) \frac{d}{d\epsilon} \bar{x}(\epsilon(t)) \right\rangle \nonumber \\
	&= \left\langle x(t) - \bar{x}(\epsilon(t)), \dot{x}(t) \right\rangle - \left\langle x(t) - \bar{x}(\epsilon(t)), \dot{\epsilon}(t) \frac{d}{d\epsilon} \bar{x}(\epsilon(t)) \right\rangle \nonumber \\
	&\leq \frac{\lambda(t)}{2} \gamma(t) \epsilon(t) (\gamma(t) \epsilon(t) - 2) \|x(t) -  \bar{x}(\epsilon(t))\|^2 \nonumber \\
	& - \left\langle x(t) - \bar{x}(\epsilon(t)), \dot{\epsilon}(t) \frac{d}{d\epsilon} \bar{x}(\epsilon(t)) \right\rangle.
	\end{align*}
	We denote $L(t)\eqdef \frac{\lambda(t)}{2} \gamma(t) \epsilon(t) (2 - \gamma(t) \epsilon(t))$. The previous inequality yields
	\begin{align*}
	\dot{\theta}(t) \leq - 2 L(t) \theta(t) - \dot{\epsilon}(t) \left\|\frac{d}{d\epsilon}\bar{x}(\epsilon(t)) \right\| \sqrt{2\theta(t)},
	\end{align*}
	where we used that $\epsilon(\cdot)$ is decreasing.
	Substituting $\varphi \eqdef \sqrt{2 \theta}$ yields $\theta = \frac{\varphi^2}{2}$ and $\dot{\theta} = \varphi \dot{\varphi}$, hence the previous inequality becomes
	\begin{align*}
	\dot{\varphi}(t) + L(t) \varphi(t) \leq - \dot{\epsilon}(t) \left\| \frac{d}{d\epsilon} \bar{x}(\epsilon(t)) \right\|.
	\end{align*}
	By Lemma \ref{diff},
	\begin{align*}
	\dot{\varphi}(t) + L(t) \varphi(t) \leq - \frac{\dot{\epsilon}(t)}{\epsilon(t)} \| \bar{x}(\epsilon(t)) \|.
	\end{align*}
	Now, we define the integrating factor $E(t) \eqdef \int_{0}^{t} L(s)\dif s$, to get 
	\begin{align*}
	\frac{\dif}{\dif t}\left(\varphi(t)\exp(E(t))\right)\leq  - \frac{\dot{\epsilon}(t)}{\epsilon(t)} \| \bar{x}(\epsilon(t)) \|\exp(E(t)).
	\end{align*}

	Hence
	\begin{align}\label{varphi}
	0\leq \varphi(t) \leq \exp(- E(t)) \left[\varphi(0) - \int_{0}^{t} \frac{\dot{\epsilon}(s)}{\epsilon(s)} \|\bar{x}(\epsilon(s))\| \exp(E(s)) \dif s \right].
	\end{align}
	If $\int_{0}^{t} \frac{\dot{\epsilon}(s)}{\epsilon(s)} \|\bar{x}(\epsilon(s))\| \exp(E(s)) \dif s$ is bounded, then $\lim_{t \to + \infty}$ $\varphi(t) = 0$; otherwise, taking into consideration $(iii)$, we employ L'H\^{o}spital rule and obtain
	\begin{align}\label{1}
	\lim\limits_{t \to + \infty} \exp(- E(t)) &\int_{0}^{t} \frac{\dot{\epsilon}(s)}{\epsilon(s)} \|\bar{x}(\epsilon(s))\| \exp(E(s)) \dif s \nonumber\\
	&= \lim\limits_{t \to + \infty} \frac{\dot{\epsilon}(t)\|\bar{x}(\epsilon(t))\|}{\epsilon^2(t)
	\lambda(t) \gamma(t) (2 - \gamma(t) \epsilon(t))} = 0,
	\end{align}
	where we used assertion $(i)$ with Lemma \ref{lem:x_eps} and assertion $(ii)$. 
	
	In conclusion, by combining \eqref{1} and $(iii)$ with \eqref{varphi}, it follows that $\varphi(t) \to 0$ as $t \to + \infty$. In particular,
	\begin{align}\label{end}
	\|x(t) - \bar{x}(\epsilon(t))\| \to 0 \text{ as } t \to + \infty.
	\end{align}
	Since
	\begin{align*}
	\|x(t) - P_{\Zer(A+B)}(0)\| \leq \|x(t) - \bar{x}(\epsilon(t))\| + \|\bar{x}(\epsilon(t)) - P_{\Zer(A+B)}(0)\|,
	\end{align*}
	the statement of the theorem follows from Lemma \ref{lem:x_eps} and \eqref{end}.
\end{proof}

\begin{remark}
	Since $\epsilon(t)$ must go to zero as $t \to + \infty$, the hypothesis $\gamma(t) \leq \frac {2 \beta} {1+ 2\beta \epsilon(t)}$ in the previous theorem implies that the stepsize function $\gamma$ is always bounded from above by $2\beta$. This corresponds to the classical assumptions in proving (weak) convergence of the discrete time forward-backward algorithm where, in order to guarantee convergence of the generated iterates, the stepsize has to be taken in the interval $(0,2\beta)$, see \cite{BC2}.
\end{remark}

\begin{remark}
	Comparing the forward-backward dynamical system \eqref{FB} with the Tikhonov regularized Krasnoselskii-Mann dynamical system \eqref{TKM} one may observe that the latter needs a constant step size function $\gamma(t)\equiv\gamma\in(0,2\beta)$. the system \eqref{FB} allows us to vary the stepsizes over time, i.e. we may choose $\gamma(\cdot)$ as a function in $t$.
\end{remark}
	

\begin{remark}
	Hypothesis (ii) of Theorem \ref{th1-fb} is fulfilled when choosing the parameter functions $\epsilon$, $\lambda$ and $\gamma$ such that ${\dot{\epsilon}(t)}/{\epsilon^2(t)} \to 0$ as $t \to + \infty$ and $\inf_{t \to + \infty} \lambda(t) > 0$, $\inf_{t \to + \infty} \gamma(t) > 0$, while Hypothesis (iii) holds true for any choice of parameter functions which satisfy $\lambda(\cdot) \gamma(\cdot) \epsilon(\cdot) \notin L^1(\R_+; \R)$ and $\epsilon(\cdot) \in L^2(\R_+; \R)$.  
	A particular instance of parameter $\beta$ and parameter functions $\epsilon$, $\lambda$ and $\gamma$ that satisfy the hypotheses of Theorem \ref{th1-fb} is given by the choice $\beta=1/2$, $\gamma (t)=1/2$, $\lambda (t) = \cos(1/t)$ and $\epsilon (t) = 1/ (1+t)^{0.6}$, $t\in [0, +\infty)$.	
\end{remark}

\section{A Tikhonov regularized forward-backward-forward dynamical system}
\label{sec:fbf}

The Tikhonov regularized forward-backward dynamical system involved a cocoercive single-valued operator $B:\scrH\to\scrH$. In order to handle more general monotone inclusion problems with less demanding regularity assumptions, Tseng constructed in \cite{TSE} a modified forward-backward scheme which shares the same weak convergence properties as the forward-backward algorithm, but is provably convergent under plain monotonicity assumptions on the involved operators $A$ and $B$. Motivated by this significant methodological improvement, we are interested in investigating a dynamical system whose trajectories strongly converge towards the minimum norm element of the set $\Zer(A+B)$, assumed nonempty, where $A : \scrH \rightrightarrows \scrH$ a maximally monotone operator, while $B : \scrH \to \scrH$ is a monotone and
$({1}/{\beta})$-Lipschitz continuous operator. The proposed dynamical system is derived from forward-backward-forward splitting algorithms coupled with a Tikhonov regularization of the single-valued operator $B$. Our starting point is the differential system
\begin{align}\label{eq:FBF}
\left\{\begin{array}{ll}
z(t) &= J_{\gamma(t) A}\big(x(t) - \gamma(t) Bx(t)\big) \\
0&= \dot{x}(t) + x(t) - z(t) - \gamma(t)\big(Bx(t) - Bz(t)\big) \\
x(0) &= x_0, \end{array}\right.
\end{align}
recently investigated in \cite{BotCseVuo18, BB8}. We assume that $\gamma : [0, + \infty) \to (0,\beta)$ is a Lebesgue measurable function and $x_0 \in\scrH$ is a given initial condition. 

Given a regularizer function $\eps:[0,+\infty)\to\R$, we modify the dynamical system \eqref{eq:FBF} to obtain the new dynamical system
\begin{align}\label{FBF}
\left\{\begin{array}{ll}
z(t) &= J_{\gamma(t) A}\big(x(t) - \gamma(t) (Bx(t) + \epsilon(t)x(t))\big)  \\
0&= \dot{x}(t) + x(t) - z(t) - \gamma(t)\big(Bx(t)-Bz(t) + \epsilon(t)(x(t) -z(t))\big) \\
x(0) &= x_0. \end{array}\right.
\end{align}

\subsection{Existence and uniqueness of strong global solutions}

In this subsection we prove the existence and uniqueness of trajectories of the dynamical system \eqref{FBF} by invoking Theorem \ref{th:existence}.


Let us define the parameterized vector field $V_{\eps,\gamma}:\scrH\to\scrH$ as 
\begin{equation}\label{eq:V}
	V_{\epsilon,\gamma}(x)\eqdef ((\Id - \gamma B_\epsilon) \circ J_{\gamma A} \circ (\Id - \gamma B_\epsilon) - (\Id - \gamma B_\epsilon))x,
\end{equation}
where $B_\epsilon \eqdef B + \epsilon \Id$. Notice that the dynamics \eqref{FBF} can be equivalently rewritten as
\begin{align*}
\left\{\begin{array}{ll}
	&\dot{x}(t) = f(t,x(t)) \\
	&x(0) = x_0,
\end{array}\right.
\end{align*}
with $f : (0,+\infty)\times \scrH \to \scrH$, given by $f(t,x)\eqdef V_{\epsilon(t),\gamma(t)}(x)$. Therefore, measurability in time and local Lipschitz continuity in the spatial variable follows after we have verified these properties for the vector field $V_{\epsilon,\gamma}(x)$. 
\begin{lemma}\label{lem:lip}
	For fixed $\epsilon \in [0,+\infty)$, let $0<\gamma < \frac {\beta} {\epsilon\beta + 1}$. Then, for all $x,y \in \scrH$, it holds
	\begin{align*}
		\|V_{\epsilon,\gamma}(x) - V_{\epsilon,\gamma}(y)\| \leq \sqrt{6} \|x-y\|.
	\end{align*}
\end{lemma}

\begin{proof}
Let $x,y \in \scrH$. For the sake of clarity, we abbreviate $C_\epsilon\eqdef \Id - \gamma B_\epsilon$ and $J\eqdef J_{\gamma A}$. \\
	First, by using the binomial formula twice we obtain
	\begin{align*}
		 \norm{V_{\epsilon,\gamma}(x) - V_{\epsilon,\gamma}(y)}^2 &=\|C_\epsilon \circ J \circ C_\epsilon x - C_\epsilon x - C_\epsilon \circ J \circ C_\epsilon y + C_\epsilon y\|^2 \\
		=&\|C_\epsilon \circ J \circ C_\epsilon x - C_\epsilon \circ J \circ C_\epsilon y\|^2 + \|C_\epsilon x - C_\epsilon y\|^2 \\
		-& 2 \langle C_\epsilon \circ J \circ C_\epsilon x - C_\epsilon \circ J \circ C_\epsilon y, Cx - Cy \rangle \\
		=&\|J \circ C_\epsilon x - J \circ C_\epsilon y\|^2 + \gamma^2 \|B_\epsilon \circ J \circ C_\epsilon x - B_\epsilon \circ J \circ C_\epsilon y\|^2 \\
		-& 2\gamma \langle J \circ C_\epsilon x - J \circ C_\epsilon y, B_\epsilon \circ J \circ C_\epsilon x - B_\epsilon \circ J \circ C_\epsilon y \rangle  \\
		+& \|C_\epsilon x - C_\epsilon y\|^2 - 2 \langle C_\epsilon \circ J \circ C_\epsilon x - C_\epsilon \circ J \circ C_\epsilon y, C_\epsilon x - C_\epsilon y \rangle.
	\end{align*}
	By invoking the $(\epsilon + 1/\beta)$-Lipschitz continuity of $B_\epsilon$ we conclude further
	\begin{align}\label{Abschaetzung}
	    &\|V_{\epsilon,\gamma}(x) - V_{\epsilon,\gamma}(y)\|^2 \nonumber \\
		\leq& \left(1 + \gamma^2\left(\epsilon + \frac{1}{\beta}\right)^2 \right) \langle C_\epsilon x - C_\epsilon y, J \circ C_\epsilon x - J \circ C_\epsilon y \rangle \nonumber \\
		&- 2\gamma \langle J \circ C_\epsilon x - J \circ C_\epsilon y, B_\epsilon \circ J \circ C_\epsilon x - B_\epsilon \circ J \circ C_\epsilon y \rangle + \|C_\epsilon x - C_\epsilon y\|^2 \nonumber \\
		&- 2 \langle C_\epsilon \circ J \circ C_\epsilon x - C_\epsilon \circ J \circ C_\epsilon y, C_\epsilon x - C_\epsilon y \rangle \nonumber \\
		=& \left(1 + \gamma^2\left(\epsilon + \frac{1}{\beta}\right)^2 - 2 \right) \langle C_\epsilon x - C_\epsilon y, J \circ C_\epsilon x - J \circ C_\epsilon y \rangle \nonumber \\
		&- 2\gamma \langle J \circ C_\epsilon x - J \circ C_\epsilon y, B_\epsilon \circ J \circ C_\epsilon x - B_\epsilon \circ J \circ C_\epsilon y \rangle + \|C_\epsilon x - C_\epsilon y\|^2 \nonumber \\
		&+ 2\gamma \langle B_\epsilon \circ J \circ C_\epsilon x - B_\epsilon \circ J \circ C_\epsilon y, C_\epsilon x - C_\epsilon y \rangle.
	\end{align}
	On the one hand, the $\epsilon$-strong monotonicity of $B_\epsilon$ yields
	\begin{align}\label{Bstrong}
		- 2\gamma \langle J \circ C_\epsilon x - J \circ C_\epsilon y, B_\epsilon \circ J \circ C_\epsilon x - B_\epsilon \circ J \circ C_\epsilon y \rangle \leq -2\gamma\epsilon \|J \circ C_\epsilon x - J \circ C_\epsilon y\|^2,
	\end{align}
	while on the other hand we deduce from the monotonicity of the resolvent and the choice of the involved parameters that
	\begin{align}\label{leq0}
		\left(\gamma^2\left(\epsilon + \frac{1}{\beta}\right)^2 - 1 \right) \langle C_\epsilon x - C_\epsilon y, J \circ C_\epsilon x - J \circ C_\epsilon y \rangle \leq 0.
	\end{align}
	Taking into account \eqref{Bstrong} and \eqref{leq0}, using the Cauchy-Schwarz inequality, the firm nonexpansiveness of the resolvent, and the $\epsilon$-strong monotonicity and the Lipschitz-continuity of $B_\epsilon$ again, we obtain from \eqref{Abschaetzung}
	\begin{align*}
		&\|V_{\epsilon,\gamma}(x) - V_{\epsilon,\gamma}(y)\|^2 \\
		\leq& -2\gamma \epsilon \|J \circ C_\epsilon x - J \circ C y\|^2 + \|C_\epsilon x - C_\epsilon y\|^2\\
		& +2\gamma \|B_\epsilon \circ J \circ C_\epsilon x - B_\epsilon \circ J \circ C_\epsilon y\| \|C_\epsilon x - C_\epsilon y\| \\
		\leq& \left(2\gamma \left(\epsilon + \frac{1}{\beta}\right) + 1\right) \|C_\epsilon x - C_\epsilon y\|^2 \\
		\leq& \left(2\gamma \left(\epsilon + \frac{1}{\beta}\right) + 1\right) (\|x-y\|^2 + \gamma^2\|B_\epsilon x - B_\epsilon y\|^2 - 2\gamma \langle x - y, B_\epsilon x - B_\epsilon y\rangle) \\
		\leq& \left(2\gamma \left(\epsilon + \frac{1}{\beta}\right) + 1\right) \left[1 + \gamma^2  \left(\epsilon + \frac{1}{\beta}\right)^2 -2 \gamma\epsilon \right] \|x-y\|^2.
	\end{align*}
	Further, by the relation imposed on $\epsilon$ and $\gamma$, we get $\gamma\epsilon\beta < \beta-\gamma$, hence 
	$$ 2\gamma \left(\epsilon + \frac{1}{\beta}\right) + 1 \leq  2 -2 \frac {\gamma}{\beta} + 2 \frac {\gamma}{\beta} + 1 = 3$$
	as well as
	$$
	1 + \gamma^2  \left(\epsilon + \frac{1}{\beta}\right)^2 -2 \gamma\epsilon = (\gamma\epsilon-1)^2 + \frac{\gamma^2}{\beta^2}(2\epsilon \beta + 1) < 1 + \left(1-\frac {\gamma}{\beta}\right)^2 + \frac{\gamma^2}{\beta^2}(2\epsilon \beta + 1) $$
		$$= 2-2\frac {\gamma}{\beta} + 2\frac{\gamma^2}{\beta^2} +2\frac {\gamma^2\epsilon}{\beta}= 1+2\frac {\gamma}{\beta^2} ( \gamma -\beta +\gamma\epsilon\beta) < 2.$$
		Consequently $\|V_{\epsilon,\gamma}(x) - V_{\epsilon,\gamma}(y)\|^2 \leq 6 \|x-y\|^2$, which yields the assertion.
\end{proof}

Based on this estimate, we obtain that
\begin{equation}
\norm{f(t,x)-f(t,y)}\leq L_{f}(t)\norm{x-y}^{2}\qquad\forall t\geq 0,x,y\in\scrH,
\end{equation}
where $L_{f}:[0,+\infty)\to\R$ is defined by
\begin{align*}
L_{f}(t)\eqdef\left(2\gamma(t) \left(\eps(t) + \frac{1}{\beta}\right) + 1\right) \left(1 + \gamma(t)\eps(t)(\gamma(t)\eps(t) + 2 \frac{\gamma(t)}{\beta} - 2)\right).
\end{align*}
Hence, by Assumption \ref{ass:lambda} it follows $L_{f}(\cdot)\in L^1_{\text{loc}}(\R_+;\R)$. We now show that $t\mapsto f(t,x)\in L^{1}_{\text{loc}}(\R_+;\scrH)$ for all $x\in\scrH$. We first establish some continuity estimates of the regularized vector field with respect to the parameters. Define the unregularized forward-backward-forward vector field
\begin{equation}\label{eq:V}
V_{\gamma}(x)\eqdef(\Id-\gamma B)\circ J_{\gamma A}\circ (\Id-\gamma B)x +\gamma Bx-x,
\end{equation}
and the residual vector field
\begin{align}\label{eq:R}
R_{\eps,\gamma}(x) \eqdef& \gamma\left(B\circ J_{\gamma A}(\Id-\gamma B)-B\circ J_{\gamma A}\circ(\Id-\gamma B_{\eps}\right)x \nonumber \\
+& \gamma\eps\left(x-J_{\gamma A}\circ(\Id-\gamma B_{\eps})x\right).
\end{align}
Simple algebra gives the decomposition
\begin{align*}
V_{\eps,\gamma}(x)=V_{\gamma}(x)+R_{\eps,\gamma}(x),
\end{align*}
From \cite{BB8}, we know that the application $\gamma\mapsto V_{\gamma}(x)$ is continuous on $(0,+\infty)$. Furthermore, \cite[Lemma 1]{BB8} gives
\begin{equation}\label{eq:V1}
\lim_{\gamma\to 0^{+}}V_{\gamma}(x)=0\ \forall x\in \scrH.
\end{equation}

\begin{lemma}\label{lem:Rzero}
If $x\in\dom A$, then
\begin{equation}
\lim_{(\eps,\gamma)\to(0,0)^{+}}R_{\eps,\gamma}(x)=0.
\end{equation}
\end{lemma}
\begin{proof}
Let  $x\in\dom A$. Nonexpansivenes gives
\begin{align*}
\norm{J_{\gamma A}(x-\gamma B_{\eps}x)-J_{\gamma A}(x-\gamma Bx)}\leq\eps\gamma\norm{x}.
\end{align*}
Since $B$ is $({1}/{\beta})$-Lipschitz, it follows
\begin{align*}
\norm{B\circ J_{\gamma A}(x-\gamma Bx)-B\circ J_{\gamma A}(x-\gamma B_{\eps}x)}\leq \frac{\eps\gamma}{\beta}\norm{x}.
\end{align*}
Furthermore,
\begin{align*}
\norm{x-J_{\gamma A}(x-\gamma B_{\eps}x)}\leq\norm{x-J_{\gamma A}(x-\gamma Bx)}+\eps\norm{x}.
\end{align*}
Summarizing the last two bounds, the triangle inequality yields that
\begin{align*}
\norm{R_{\eps,\gamma}(x)}\leq\frac{\gamma^{2}\eps}{\beta}\norm{x}+\gamma\eps\norm{x-J_{\gamma A}(x-\gamma Bx)}+\gamma\eps^{2}\norm{x}.
\end{align*}
By \cite[Proposition 6.4]{PR4}, we know that $\lim_{\gamma\to 0^{+}}J_{\gamma A}(x-\gamma Bx)=\Pr_{\cl\dom A}(x)=x$. From here the result easily follows.
\end{proof}

\begin{lemma}\label{lem:Vcont}
For all $x\in\dom A$, we have
\begin{equation}
\lim_{(\eps,\gamma)\to (0,0)}V_{\eps,\gamma}(x)=0.
\end{equation}
\end{lemma}
\begin{proof}
We just have to combine \eqref{eq:V1} with the decomposition $V_{\eps,\gamma}(\cdot)=V_{\gamma}(\cdot)+R_{\eps,\gamma}(\cdot)$ and Lemma \ref{lem:Rzero}.
\end{proof}

Define the set
\begin{equation}
\Theta\eqdef \left\{(\eps,\gamma)\in\R^{2}_{++}\vert \gamma < \frac {\beta} {\eps\beta + 1}\right\}.
\end{equation}
By nonexpansivenes of the resolvent operator $J_{\gamma A}$ and continuity of $B$, it follows that the map $(\eps,\gamma)\mapsto R_{(\eps,\gamma)}(x)$ is continuous. Furthermore, we can extend it continuously to the closure of the parameter space $\Theta$, denoted as $\bar{\Theta}$, as Lemma \ref{lem:Vcont} shows. This allows us to prove the local boundedness of the vector field.
\begin{lemma}\label{lem:growth}
For all $(\eps,\gamma)\in\Theta$ and all $x\in\scrH$, there exists $K>0$ such that
\begin{equation}
\norm{V_{\eps,\gamma}(x)}\leq K(1+\norm{x}).
\end{equation}
\end{lemma}
\begin{proof}
Fix $\bar{x}\in\dom A$. By Lemma \ref{lem:Vcont}, the application $(\eps,\gamma)\mapsto f(\eps,\gamma, \cdot)$ can be continuously extended to the set
\begin{align*}
\bar{\Theta} \eqdef \left\{(\eps,\gamma)\in\R^{2}_{+}\vert \gamma < \frac {\beta} {\epsilon\beta + 1}\right\}.
\end{align*}
Hence, there exists a constant $M>0$ such that $\norm{f(\eps,\gamma, \bar{x})}\leq M$ for all $(\eps,\gamma)\in\Theta$. Furthermore, using Lemma \ref{lem:lip}, we get
\begin{align*}
\norm{V_{\eps,\gamma}(x)}&\leq \norm{V_{\eps,\gamma}(\bar{x})}+\norm{V_{\eps,\gamma}(x)- V_{\eps,\gamma}(\bar{x})}\\
&\leq M+\sqrt{3}\norm{x-\bar{x}}\\
&\leq K(1+\norm{x})
\end{align*}
where we can choose $K\eqdef\max\{\sqrt{3},M+\sqrt{3}\norm{\bar{x}}\}$.
\end{proof}
All these estimates allow us now to prove existence and uniqueness of solutions to the dynamical system \eqref{FBF}.

\begin{theorem}\label{prop:existence}
Let $(\eps,\gamma):[0,+\infty)\to\Theta$ be measurable. Then, for each $x_{0}\in\scrH$, there exists a unique strong solution $t\mapsto x(t)$, $t\geq 0$, of \eqref{FBF}.
\end{theorem}
\begin{proof}
We verify conditions $(f1)-(f4)$ of Theorem \ref{th:existence} for the map $f(t,x)= V_{\eps(t),\gamma(t)}(x)$. Conditions $(f1),(f2)$ follow from the integrability assumptions on the functions $\eps(t),\gamma(t)$. For all $x,y\in\scrH$ and all $t\geq 0$ we have
\begin{align*}
&\norm{f(t,x)-f(t,y)}\leq \sqrt{3}\norm{x-y}, \text{ and }\\
&\norm{f(t,x)}\leq K(1+\norm{x}).
\end{align*}
Hence, $(f3),(f4)$ follow as well. 
\end{proof}

\subsection{Convergence of the trajectories}

In order to show strong convergence of the strong global solution of \eqref{FBF} towards the minimum norm element of $\Zer(A+B)$, we need some additional preparatory results.

\begin{lemma}\label{Bt}
	For almost all $t \in [0, +\infty)$, we have
	\begin{align*}
	0 \leq &\|x(t) - \bar{x}(\epsilon(t))\|^2 - \|x(t) - z(t)\|^2 - (1 + 2\epsilon(t)\gamma(t)) \|z(t) - \bar{x}(\epsilon(t))\|^2 \\
	&+ 2\gamma(t) \langle B_\eps(t) z(t) - B_\eps(t) x(t), z(t) - x_{\epsilon(t)} \rangle
	\end{align*}
\end{lemma}

\begin{proof}
	First, we observe that the first line in \eqref{FBF} can be equivalently rewritten as
	\begin{align}\label{1}
	\frac{x(t) - z(t)}{\gamma(t)} - B_{\eps(t)}x(t) \in Az(t),
	\end{align}
	hence
	\begin{align*}
	\frac{x(t) - z(t)}{\gamma(t)} + B_{\eps(t)}z(t) - B_{\eps(t)}x(t) = - \frac{\dot{x}(t)}{\gamma(t)} \in (A + B_{\eps(t)})z(t).
	\end{align*}
	On the other hand $0 \in \gamma(t) (A + B_{\eps(t)})\bar{x}(\epsilon(t))$. Using the $\epsilon(t)$-strong monotonicity of $A + B_{\eps(t)}$ yields
	\begin{align*}
	2 \epsilon(t) \gamma(t) \|z(t) - \bar{x}(\epsilon(t))\|^2 &\leq 2 \langle x(t) - z(t) + \gamma(t) B_{\eps(t)} z(t) \\
	& - \gamma(t) B_{\eps(t)} x(t), z(t) - \bar{x}(\epsilon(t)) \rangle \\
	& = \|x(t) - \bar{x}(\epsilon(t))\|^2 - \|x(t) - z(t)\|^2 - \|z(t) - \bar{x}(\epsilon(t))\|^2 \\
	& + 2 \gamma(t) \langle B_{\eps(t)} z(t) - B_{\eps(t)}x(t), z(t) - \bar{x}(\epsilon(t)) \rangle.
	\end{align*}
	This shows the assertion.
\end{proof}


\begin{lemma}\label{xdot}
	Let $t\mapsto x(t)$, $t\geq 0$, be the strong global solution of \eqref{FBF}. Then, for almost all $t \in [0,+\infty)$
	\begin{align*}
		\langle x(t) -\bar{x}(\epsilon(t)), \dot{x}(t) \rangle & \leq \left( \gamma(t) \epsilon(t) + \frac{\gamma(t)- 1}{\beta} \right) \|x(t) - z(t)\|^2\\
		& - \epsilon(t) \gamma(t) \|z(t) - \bar{x}(\epsilon(t))\|.
	\end{align*}
\end{lemma}

\begin{proof}
	We have for almost all $t \in [0,+\infty)$
	\begin{align*}
		2 \langle x(t) - \bar{x}(\epsilon(t)), \dot{x}(t) \rangle &= 2 \langle x(t) - \bar{x}(\epsilon(t)), z(t) - x(t)  \rangle + 2 \gamma(t) \langle x(t)\\
		& \bar{x}(\epsilon(t)), B_{\eps(t)} x(t)  - B_{\eps(t)} z(t)  \rangle\\
		& = \|z(t) - \bar{x}(\epsilon(t))\|^2 - \|x(t) - \bar{x}(\epsilon(t))\|^2  \\
		& - \|z(t) - x(t)\|^2 + 2\gamma(t) \langle x(t) - \bar{x}(\epsilon(t)), B_{\eps(t)} x(t) - B_{\eps(t)} z(t) \rangle.
	\end{align*}
	By Lemma \ref{Bt}, for almost all $t \in [0,+\infty)$ one has
	\begin{align*}
		\|z(t) - \bar{x}(\epsilon(t))\|^2 - \|x(t) - \bar{x}(\epsilon(t))\|^2 &\leq - \|x(t) - z(t)\|^2 - 2\epsilon(t)\gamma(t) \|z(t) - \bar{x}(\epsilon(t))\|^2 \\
		&~~~~ + 2\gamma(t) \langle B_{\eps(t)} z(t) - B_{\eps(t)} x(t), z(t) - \bar{x}(\epsilon(t)) \rangle,
	\end{align*}
	therefore, by using that $B_{\eps(t)}$ is $(\epsilon(t) + 1/\beta)$-Lipschitz continuous it holds for almost all $t \in [0,+\infty)$
	\begin{align*}
		2 \langle x(t) - \bar{x}(\epsilon(t)), \dot{x}(t) \rangle &\leq -2\|x(t) - z(t)\|^2 - 2\epsilon(t)\gamma(t) \|z(t) - \bar{x}(\epsilon(t))\|^2 \\
		&~~~~+ 2\gamma(t) \langle B_{\eps(t)} z(t) - B_{\eps(t)} x(t), z(t) - x(t) \rangle 
			\end{align*}
$$\leq -2\left(1 - \gamma(t) \epsilon(t) - \frac{\gamma(t)}{\beta} \right) \|x(t) - z(t)\|^2 - 2\epsilon(t)\gamma(t) \|z(t) - \bar{x}(\epsilon(t))\|^2.$$			
\end{proof}

The convergence statement follows.

\begin{theorem}\label{FBFconvergence}
	Let $t\mapsto x(t)$, $t\geq 0$, be the strong solution of \eqref{FBF}. Suppose that $\gamma(t) < \frac {\beta} {\epsilon(t)\beta + 1}$ for all $t \in [0,+\infty)$ and that the following properties are fulfilled
	\begin{align*}
	(i) ~&\epsilon \text{ is absolutely continuous and } \epsilon(t) \text{ decreases to } 0 \text{ as } t \to + \infty, \\[5pt]
	(ii) ~&\frac{\dot{\epsilon}(t)}{\epsilon^2(t)\gamma(t)(\beta(1-\gamma (t)\epsilon(t))-\gamma (t))} \to 0 \text{ as } t \to + \infty, \\[5pt]
	(iii) ~&\int_{0}^{+\infty} \frac{\gamma(t) \epsilon(t)\big(\beta - \beta\gamma(t) \epsilon(t) - \gamma(t)\big)} {\beta \gamma(t) \epsilon(t)+\beta + \gamma(t)} dt = + \infty.
	\end{align*}
	Then $x(t) \to P_{\Zer(A+B)}(0)$ and $z(t) \to P_{\Zer(A+B)}(0)$ as $t \to + \infty$.
\end{theorem}

\begin{proof}
Define $\theta(t)\eqdef \frac{1}{2} \|x(t) -\bar{x}(\epsilon(t))\|^2$, $t\geq 0$. Then, by using Lemma \ref{xdot}, for almost all $t\geq 0$
\begin{align}\label{theta}
	\dot{\theta}(t) &= \left\langle x(t) -\bar{x}(\epsilon(t)), \dot{x}(t) - \dot{\epsilon}(t) \frac{d}{d\epsilon} \bar{x}(\epsilon(t)) \right\rangle \nonumber \\
	&= \left\langle x(t) - \bar{x}(\epsilon(t)), \dot{x}(t) \right\rangle - \left\langle x(t) - \bar{x}(\epsilon(t)), \dot{\epsilon}(t) \frac{d}{d\epsilon} \bar{x}(\epsilon(t)) \right\rangle \nonumber \\
	&\leq - \left(1 - \gamma(t) \epsilon(t) - \frac{\gamma(t)}{\beta} \right) \|x(t) - z(t)\|^2 - \epsilon(t) \gamma(t) \|z(t) - \bar{x}(\epsilon(t))\| \nonumber \\
	&- \left\langle x(t) - \bar{x}(\epsilon(t)), \dot{\epsilon}(t) \frac{d}{d\epsilon} \bar{x}(\epsilon(t)) \right\rangle.
\end{align}
Further, for almost all $t\geq 0$, by $\epsilon(t)$-strong monotonicity of $A + B_t$ one has
\begin{align*}
\epsilon(t) \|z(t) - \bar{x}(\epsilon(t))\|^2 &\leq \left\langle \frac{x(t) - z(t)}{\gamma(t)} + B_{\eps(t)}z(t)-B_{\eps(t)} x(t) , z(t) - \bar{x}(\epsilon(t)) \right\rangle,
\end{align*}
hence by Cauchy-Schwarz inequality, employing the $(\epsilon(t) + 1/\beta)$-Lipschitz continuity of $B_{\eps(t)}$ and rearranging terms, for almost all $t\geq 0$ it holds
\begin{align*}
	\|z(t) - \bar{x}(\epsilon(t))\| \leq \left( 1 + \frac{1}{\gamma(t) \epsilon(t)} + \frac {1}{\beta \epsilon(t)}\right) \|x(t) - z(t)\|.
\end{align*}
In particular, for almost all $t\geq 0$ one has
\begin{align*}
	\|x(t) - \bar{x}(\epsilon(t))\| &\leq \|x(t) - z(t)\| + \|z(t) - \bar{x}(\epsilon(t))\|\\ 
	& \leq \left( 2 + \frac{1}{\gamma(t) \epsilon(t)} + \frac {1}{\beta \epsilon(t)}\right) \|x(t) - z(t)\|,
\end{align*}
which is equivalent to
\begin{align}\label{-xtzt}
	- \|x(t) - z(t)\| \leq - \frac{\beta \gamma(t) \epsilon(t)}{2\beta \gamma(t) \epsilon(t)+\beta + \gamma(t)} \|x(t) - \bar{x}(\epsilon(t))\|
\end{align}
for almost all $t\geq 0$. Inserting \eqref{-xtzt} in \eqref{theta} and dropping the second (nonpositive) term on the right hand side yields, after denoting 
$$L(t)\eqdef \left(1 - \gamma(t) \epsilon(t) - \frac{\gamma(t)}{\beta} \right) \frac{\beta \gamma(t) \epsilon(t)}{2\beta \gamma(t) \epsilon(t)+\beta + \gamma(t)},\ t\geq 0,$$ we see
\begin{align*}
	\dot{\theta}(t) &\leq - 2 L(t) \theta(t) - \left\langle x(t) - \bar{x}(\epsilon(t)), \dot{\epsilon}(t) \frac{d}{d\epsilon} \bar{x}(\epsilon(t)) \right\rangle \\
	&\leq - 2 L(t) \theta(t) - \dot{\epsilon}(t) \left\|\frac{d}{d\epsilon} \bar{x}(\epsilon(t)) \right\| \sqrt{2\theta(t)},
\end{align*}
for almost all $t\geq 0$, where in the second inequality we used that $\epsilon(\cdot)$ is decreasing, thus $\dot{\epsilon}(\cdot)$ is nowhere positive. From here, we can repeat the arguments from the proof of Theorem \ref{th1-fb}, mutatis mutandis, to obtain the desired result.

Analogous to the proof of \cite[Theorem 2]{BB8} one can show, by replacing in the demonstrations of the intermediate results $B$ by $B_{\epsilon(t)}$ and taking into consideration the absolute continuity of $\epsilon$, that $\lim_{t\to +\infty}(x(t)-z(t))=0$, hence $z(t) \to P_{\Zer(A+B)}(0)$ as $t \to + \infty$ as well.  

\end{proof}

\begin{remark}
If $\sup_{t\to +\infty} \gamma(t) < \beta$ one can replace assertion $(ii)$ of the previous theorem with
\begin{align*}
(ii') ~&\frac{\dot{\epsilon}(t)}{\epsilon^2(t)\gamma(t)} \to 0 \text{ as } t \to + \infty.
\end{align*}
\end{remark}

\begin{remark}
If we choose, for example, $\epsilon(t) = {1}/{(1+t)^{0.5}}$ and $\gamma(t) \equiv \gamma \in (0,\beta)$ constant, symbolic computation with Mathematica shows that in this case assertion $(iii)$ holds (as well as assertions $(i)$ and $(ii)$ by choice of $\epsilon(\cdot)$).
\end{remark}

\begin{remark}
	The strong convergence of the trajectories of a forward-backward-forward dynamical system was achieved in \cite[Theorem 3]{BB8} under more demanding hypothesis involving the strong monotonicity of sum of the involved operators.
\end{remark}


\section{Numerical illustrations}
\label{sec:illustrations}

In this section we are going to illustrate by some numerical experiments the theoretical results we achieved. More precisely, we show how adding a Tikhonov regularization term in the considered dynamical systems influences the asymptotic behavior of their trajectories. All code was written in {\sc Matlab} using the {\sc ode15s} function for solving ordinary differential equations.

\subsection{Application to a split feasibility problem}
\label{sec:split}
For our first example we consider the following \textit{split feasibility problem} in $\R^2$
\begin{align}\label{SFP}
\text{find } x \in \R^2 \text{ such that } x \in C \text{ and } Lx \in Q, \tag{SFP}
\end{align}
where $C$ and $Q$ are nonempty, closed and convex subsets of $\R^2$ and $L : \R^2 \to \R^2$ a bounded linear operator. For this purpose, we first notice that \eqref{SFP} can be equivalently rewritten as
\begin{align*}
\min_{x \in C} \left\{ \frac{1}{2} \| Lx - \Pr_{Q} (Lx) \|^2 \right\}.
\end{align*}
The necessary and sufficient optimality condition for this problem yields
\begin{align}\label{OCSFP}
	\text{find } x \in \R^2 \text{ such that } 0 \in \NC_C(x) + L^\ast \circ (\Id - \Pr_Q) \circ Lx.
\end{align}
We approach \eqref{SFP} by the two Tikhonov regularized forward-backward dynamical systems we developed in this paper as well as by an unregularized version and compare the trajectories.
In order to apply the forward-backward dynamical systems to the monotone inclusion problem \eqref{OCSFP}, we set $A\eqdef\NC_C$ and $B\eqdef \nabla(\frac{1}{2} \| L(\cdot) - \Pr_{Q} (L(\cdot)) \|^2) =  L^\ast \circ (\Id - \Pr_Q) \circ L$. It holds for $x, y \in \R^2$
\begin{align*}
\|Bx - By\| \leq \|L\|^2 \|x - y\| + \|L\| \|\Pr_QLx - \Pr_QLy\| \leq 2\|L\|^2 \|x - y\|,
\end{align*}
i.e. $B$ is Lipschitz continuous with constant $2\|L\|^2$ and due to the Baillon-Haddad theorem $B$ is $({1}/({2\|L\|^2}))$-cocoercive. Hence Theorem \ref{FBout} and Theorem \ref{th1-fb} as well as the convergence statement \cite[Theorem 6]{BC7} for the non-regularized forward-backward dynamical system can be employed for \eqref{OCSFP} writen by means of the operators $A$ and $B$. By taking into account that $J_A = \Pr_C$ we obtain the following dynamical systems
\begin{align*}
&\left\{\begin{array}{ll} \dot{x}(t) &= \lambda(t) [\Pr_C(x(t) - \gamma B(x(t))) - x(t)] \\
x(0) &= x_0, \end{array}\right.  \label{FBWO} \tag{FB} \\[5pt]
&\left\{\begin{array}{ll} \dot{x}(t) &= \lambda(t) [\Pr_C(x(t) - \gamma B(x(t))) - x(t)] - \epsilon(t)x(t) \\
x(0) &= x_0, \end{array}\right.  \label{FBOR} \tag{FBOR} \\[5pt]
\end{align*}
and
\begin{align}
&\left\{\begin{array}{ll} \dot{x}(t) &= \lambda(t) [\Pr_C[ x(t) - \gamma ( B(x(t)) + \epsilon(t)x(t) )] - x(t)] \\
x(0) &= x_0, \end{array}\right.  \label{FBIR} \tag{FBIR}
\end{align}
that are special cases of \eqref{FBo}, \eqref{FBx} and \eqref{FB}, respectively. For the implementation we take $C \eqdef B_1(0)$ the open ball with center $0$ and radius $1$ in $\R^2$ and $Q \eqdef \{ x \in \R^2 : 3x_1 - x_2 = 0\}$ a linear subspace. Moreover, we define
\[L \eqdef \left( \begin{array}{rr}
	1 & -1 \\
	1 & 1
\end{array}\right)\]
and set $x_0 = (-3,3)^{\top} \in \R^2$ as starting point. Obviously, $\|L\| = \sqrt{2}$.
According to \cite[Proposition 29.10 and Example 29.18]{BC2}, the projections onto the sets $C$ and $Q$ are given by
\[\Pr_C(x) = \left\{\begin{array}{ll} \frac{x}{\|x\|}, & \|x\| > 1, \\[3pt]
x, & \text{else}, \end{array}\right. \]
and
\[\Pr_Q(x) = x + \frac{\eta - \langle x, u \rangle}{\|u\|^2}u ,\]
with $u = (3,-1)^{\top} \in \R^2$ and $\eta = 0$, respectively. Further, we choose $\epsilon(t) \eqdef {1}/({(1+t)^{0.5}})$ as the Tikhonov regularization function.
For different choices of the parameters $\lambda(t) \equiv \lambda > 0$ and $\gamma > 0$ the resulted trajectories of the dynamical systems \eqref{FBWO}, \eqref{FBIR} and \eqref{FBOR} are displayed in Figures \ref{02505} to \ref{051}.
\begin{figure}[htbp]
	\centering
	\includegraphics[width=0.32\textwidth]{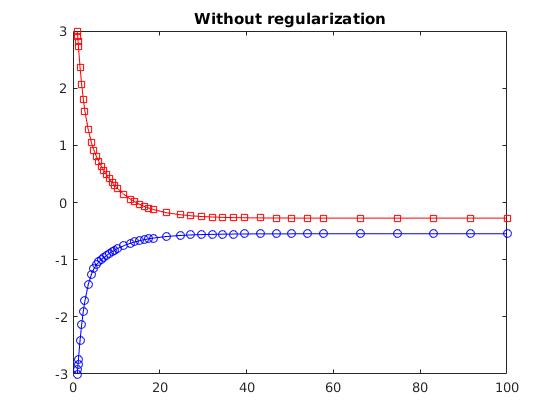}
	\includegraphics[width=0.32\textwidth]{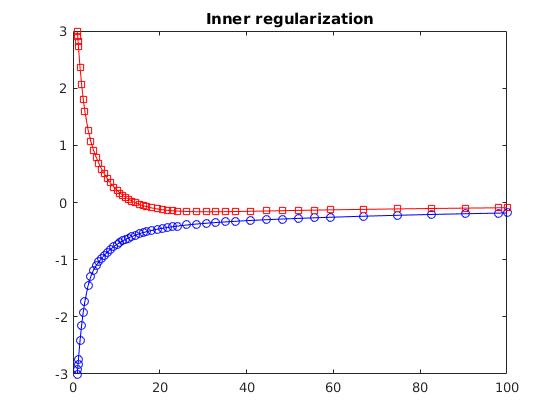}
	\includegraphics[width=0.32\textwidth]{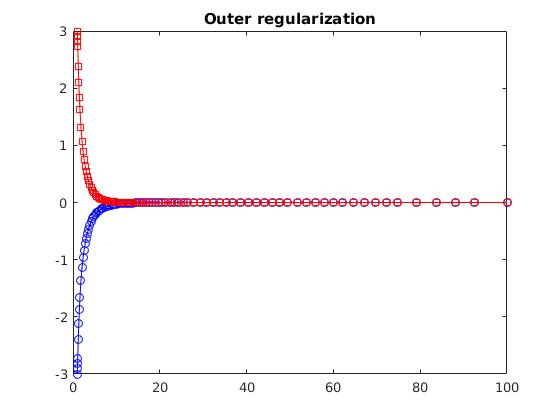}
	\caption{Trajectories of \eqref{FBWO}, \eqref{FBIR} and \eqref{FBOR} for $\lambda = 0.5$ and $\gamma = 0.15$}\label{02505}
\end{figure}

\begin{figure}[htbp]
	\centering
	\includegraphics[width=0.32\textwidth]{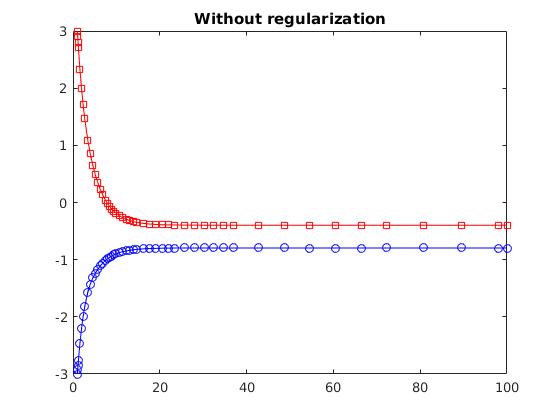}
	\includegraphics[width=0.32\textwidth]{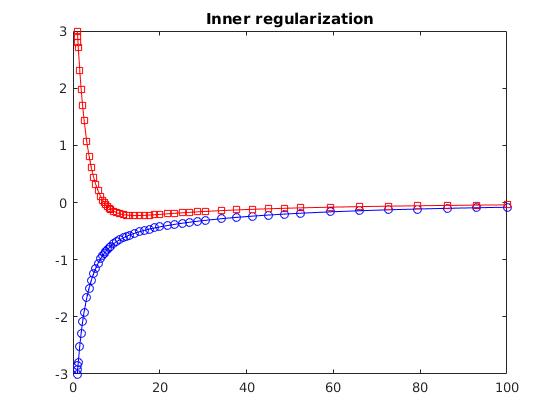}
	\includegraphics[width=0.32\textwidth]{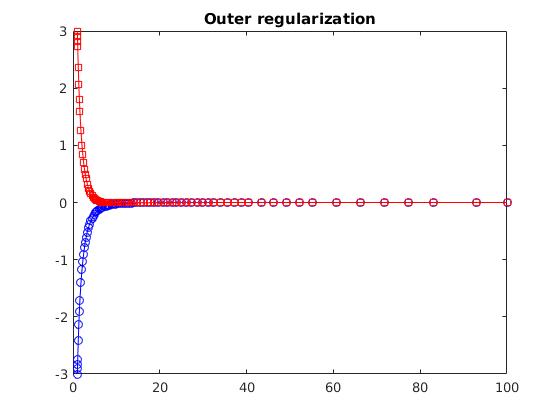}
	\caption{Trajectories of \eqref{FBWO}, \eqref{FBIR} and \eqref{FBOR} for $\lambda = 0.5$ and $\gamma = 0.3$}\label{0505}
\end{figure}

\begin{figure}[htbp]
	\centering
	\includegraphics[width=0.32\textwidth]{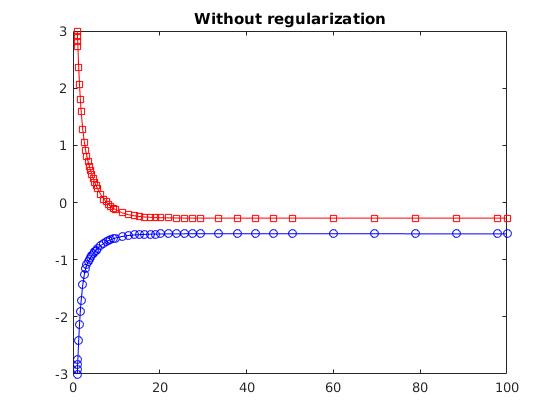}
	\includegraphics[width=0.32\textwidth]{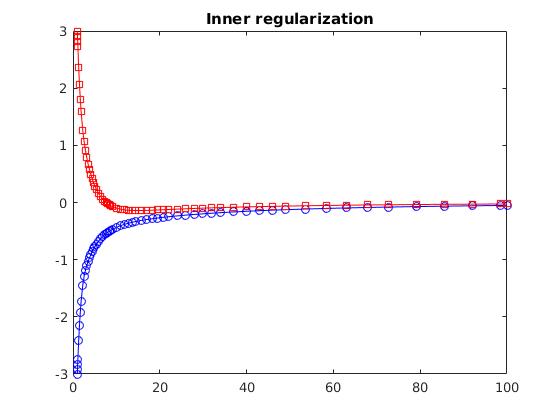}
	\includegraphics[width=0.32\textwidth]{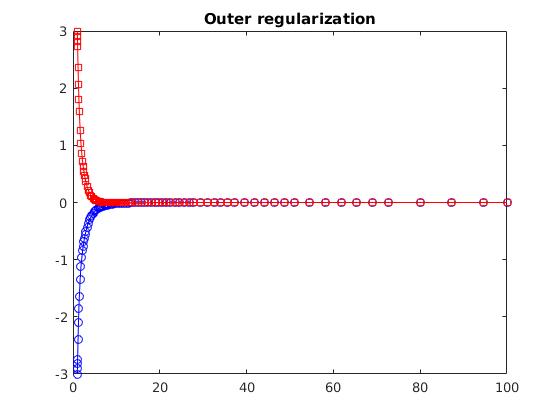}
	\caption{Trajectories of \eqref{FBWO}, \eqref{FBIR} and \eqref{FBOR} for $\lambda = 1$ and $\gamma = 0.15$}\label{0251}
\end{figure}

\begin{figure}[htbp]
	\centering
	\includegraphics[width=0.32\textwidth]{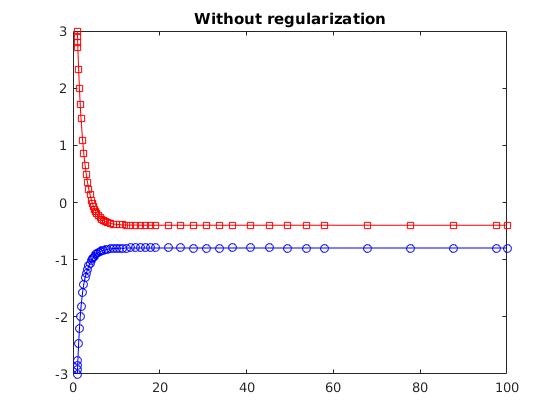}
	\includegraphics[width=0.32\textwidth]{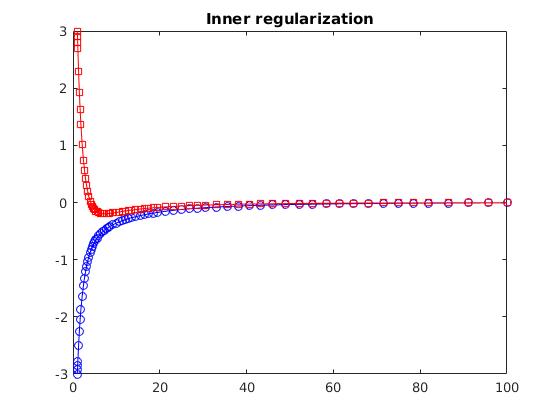}
	\includegraphics[width=0.32\textwidth]{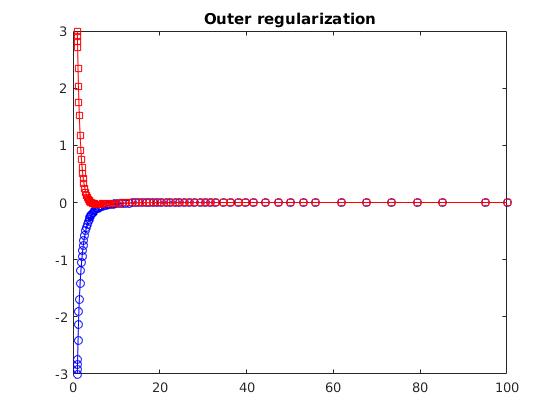}
	\caption{Trajectories of \eqref{FBWO}, \eqref{FBIR} and \eqref{FBOR} for $\lambda = 1$ and $\gamma = 0.3$}\label{051}
\end{figure}
One observes the following: while the trajectories of the unregularized system \eqref{FBWO} approach a solution of \eqref{SFP} with positive norm, the regularized dynamical systems \eqref{FBIR} and \eqref{FBOR} generate trajectories which converge to the minimum norm solution $(0,0)^{\top} \in \R^2$ of \eqref{SFP}. Furthermore, for small parameters $\lambda$ and $\gamma$ the outer regularization \eqref{FBOR} acts more aggressively than the inner regularization \eqref{FBIR}, leading to a faster convergence of the trajectories of \eqref{FBOR}. In contrast, the trajectories of \eqref{FBIR} are gently guided to the minimum norm solution and one can recognize the shape of the unregularized trajectories generated by \eqref{FBWO}.
However, for larger $\lambda$ and $\gamma$, the differences between the trajectories generated by the two Tikhonov regularized systems seem to fade.

\subsection{Application to a variational inequality}

For the second numerical illustration, this time of the forward-backward-forward splitting scheme, we consider the variational inequality
\begin{align}\label{VI}
\text{find } x \in \R^3 \text{ such that } \langle B(x), y - x \rangle \geq 0~\forall y \in C, \tag{VI}
\end{align}
where $B : \R^3 \to \R^3$ is a Lipschitz continuous mapping and $C \subseteq \R^3$ a nonempty, closed and convex set. To attach a forward-backward-forward dynamical system to this problem, we note that \eqref{VI} can be equivalently rewritten as the monotone inclusion
\begin{align}\label{MonIncl}
\text{find } x \in \R^3 \text{ such that } 0 \in B(x) + \NC_C(x).
\end{align}
Hence, by setting $A \eqdef \NC_C$ and taking into consideration that $J_A = \Pr_C$, the Tikhonov regularized forward-backward-forward dynamical system \eqref{FBF} associated to problem \eqref{MonIncl} reads as
\begin{align}\label{FBFVI}\tag{$FBFR$}
\left\{\begin{array}{ll}
z(t) &= \Pr_C[x(t) - \gamma(t) (Bx(t) + \epsilon(t)x(t))]  \\
0&= \dot{x}(t) + x(t) - z(t) - \gamma(t)[Bx(t)-Bz(t) + \epsilon(t)(x(t) -z(t))] \\
x(0) &= x_0. \end{array}\right.
\end{align}
For the implementation we specify
\[B \eqdef \left( \begin{array}{ccc}
0 & 0.1 & 0.5 \\ -0.1 & 0 & -0.4 \\ -0.5 & 0.4 & 0
\end{array}\right)\]
which defines a linear operator and $C\eqdef \{x \in \R^3 : 3x_1 - x_2 + 1 = 0\}$. Since $B$ is skew-symmetric (i.e. $B^{\top} = - B$), it can not be cocoercive, hence our theoretical results on the forward-backward dynamical systems cannnot be used for solving \eqref{MonIncl}. However, since $B$ is Lipschitz continuous with constant $\|B\| \approx 0.64807$ we can apply Theorem \ref{FBFconvergence} for finding a solution to \eqref{MonIncl}. Similarly as in the previous subsection, according to \cite[Example 29.18]{BC2} the projection onto $C$ is given by
\[\Pr_C(x) = x + \frac{\eta - \langle x, u \rangle}{\|u\|^2}u ,\]
with $u = (3,-1,1)^{\top} \in \R^3$ and $\eta = 0$. We choose $x_0 \eqdef (-2, 4, -2)^{\top}$ as starting point and $\epsilon(t)\eqdef \frac {1}{(1+t)^\beta}$ with $\beta \in [0,1)$ as Tikhonov regularization function. We call $\beta$ the \textit{Tikhonov regularization parameter} and note that the choice $\beta = 0$ corresponds to the unregularized system \eqref{eq:FBF} as investigated in \cite{BB8}.
The trajectories of \eqref{FBFVI} for the choices of regularization parameters $\beta \in \{0,0.1,0.5,0.9\}$ and step sizes $\gamma \in \{0.2,0.5\}$ are pictured in Figures \ref{FBFVI02} and \ref{FBFVI05}, respectively.
\begin{figure}[htbp]
	\centering
	\includegraphics[width=0.49\textwidth]{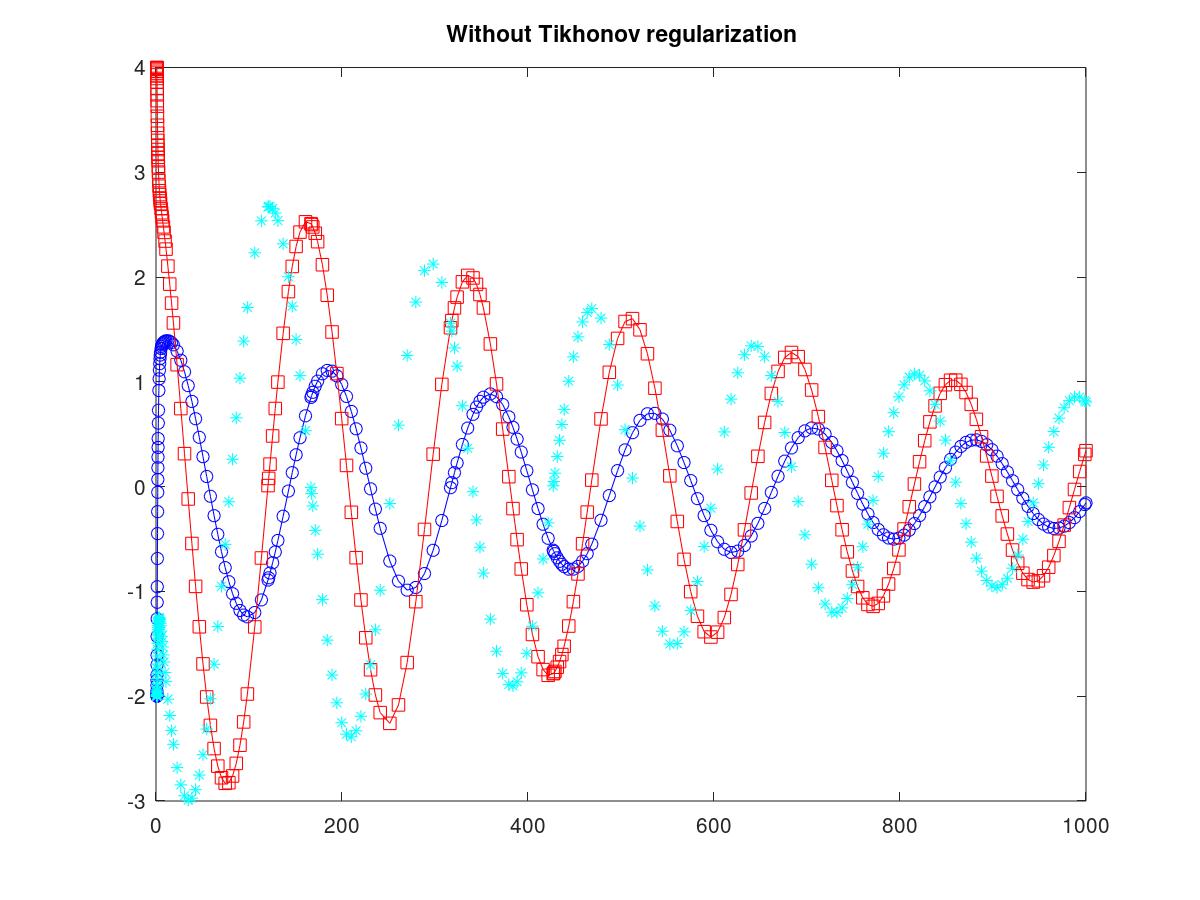}
	\includegraphics[width=0.49\textwidth]{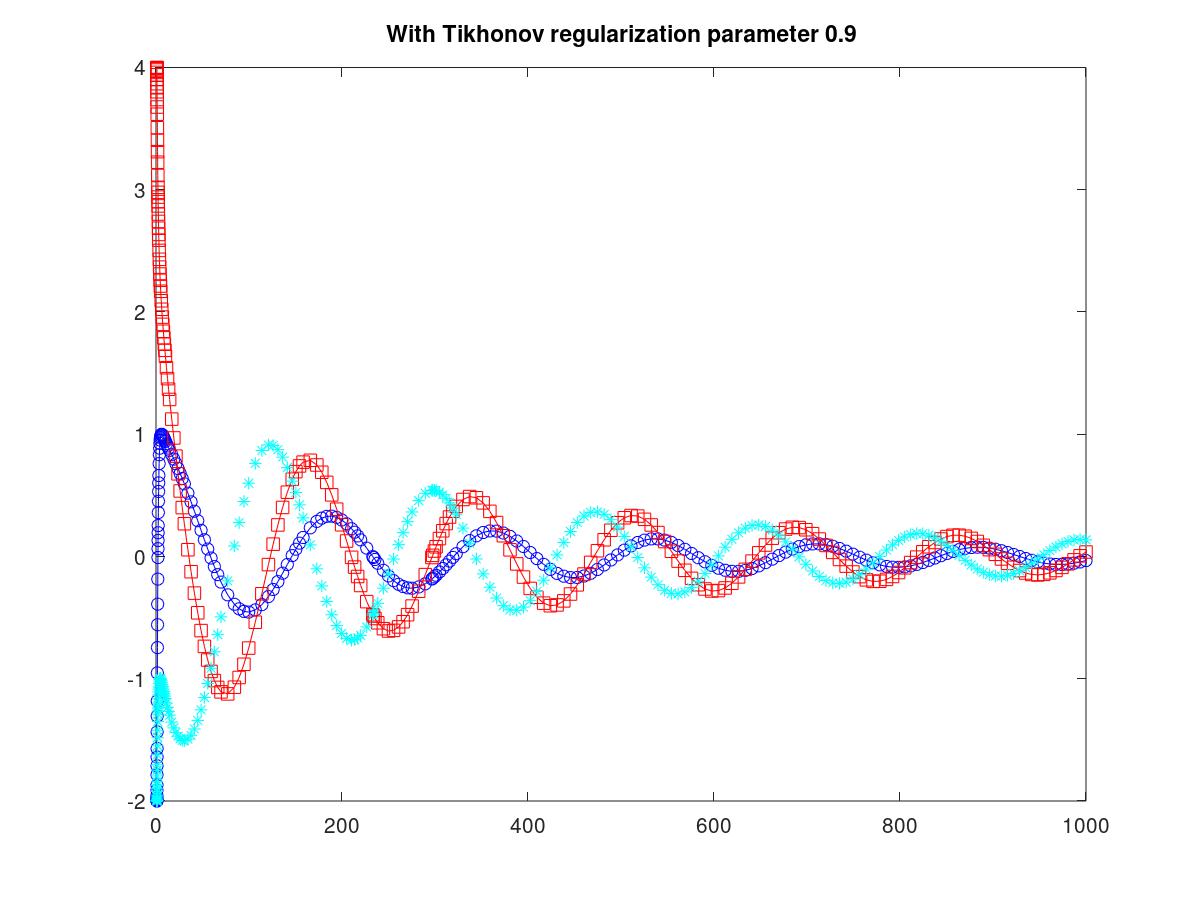}
	\includegraphics[width=0.49\textwidth]{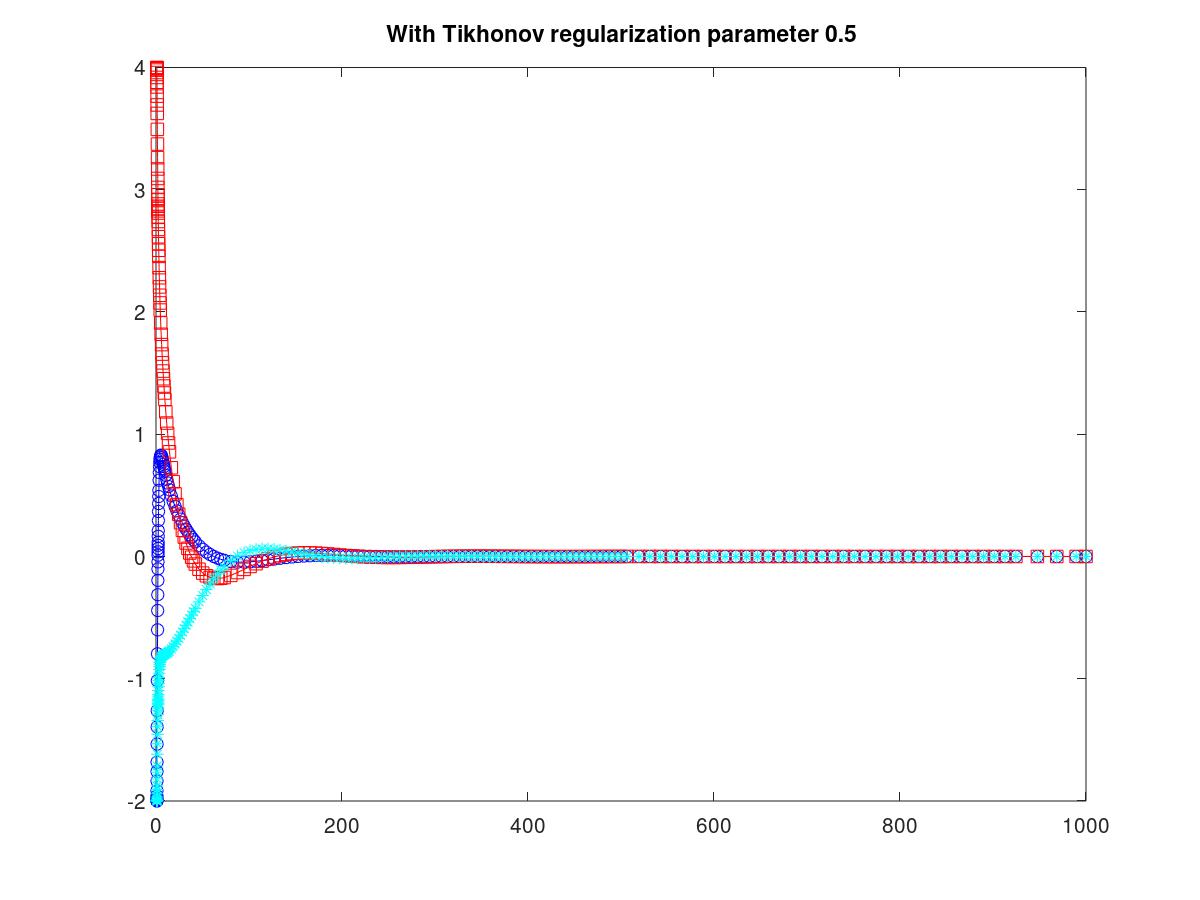}
	\includegraphics[width=0.49\textwidth]{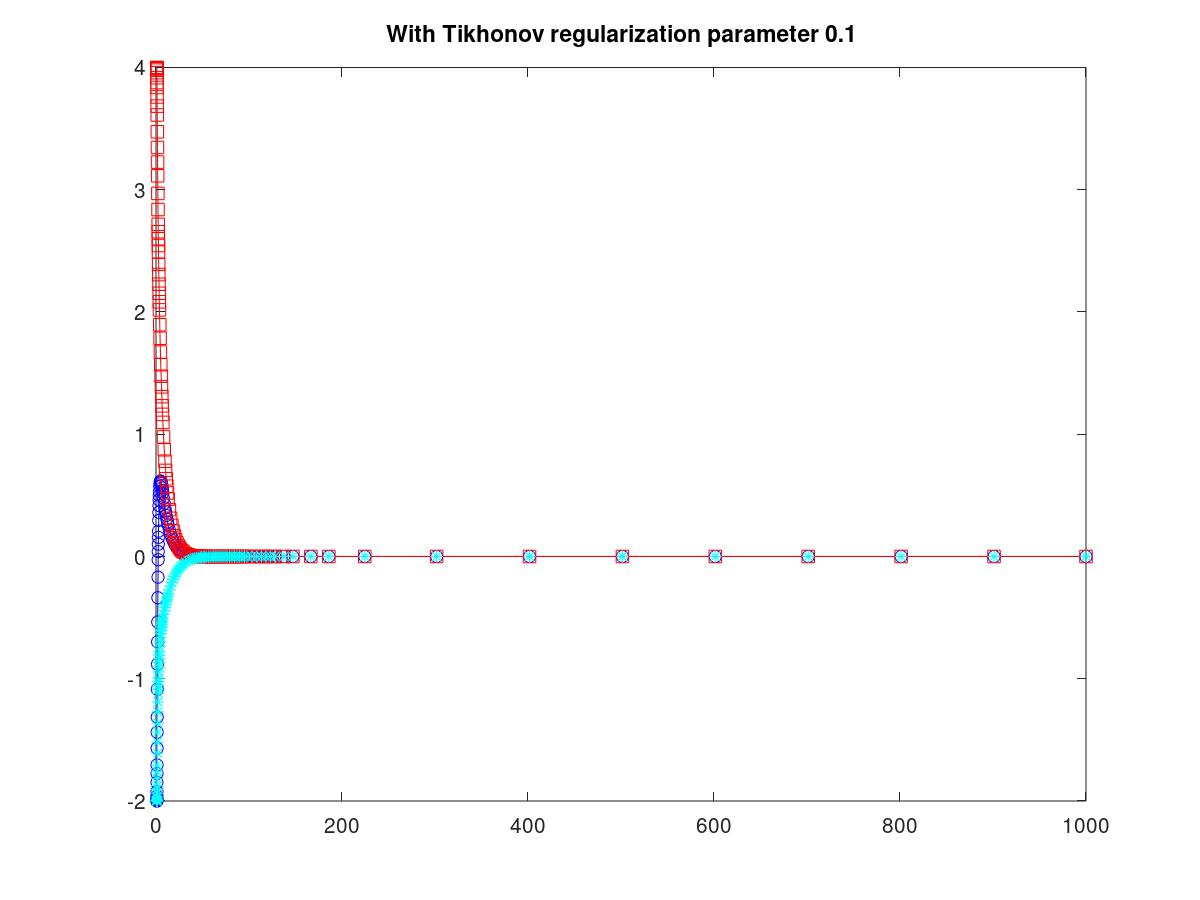}
	\caption{Trajectories of \eqref{FBFVI} for regularization parameters $\beta \in \{0,0.1,0.5,0.9\}$ and $\gamma = 0.2$}\label{FBFVI02}
\end{figure}

\begin{figure}[htbp]
	\centering
	\includegraphics[width=0.49\textwidth]{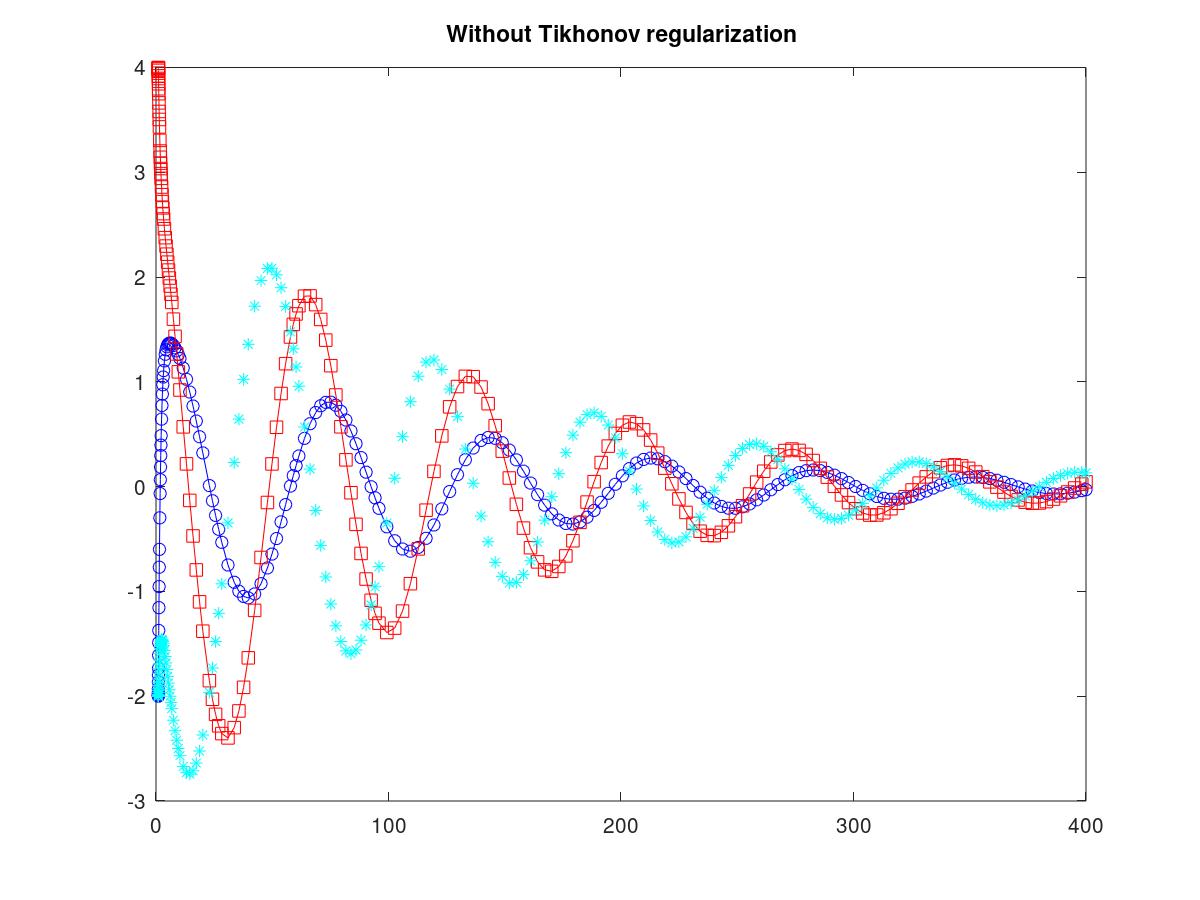}
	\includegraphics[width=0.49\textwidth]{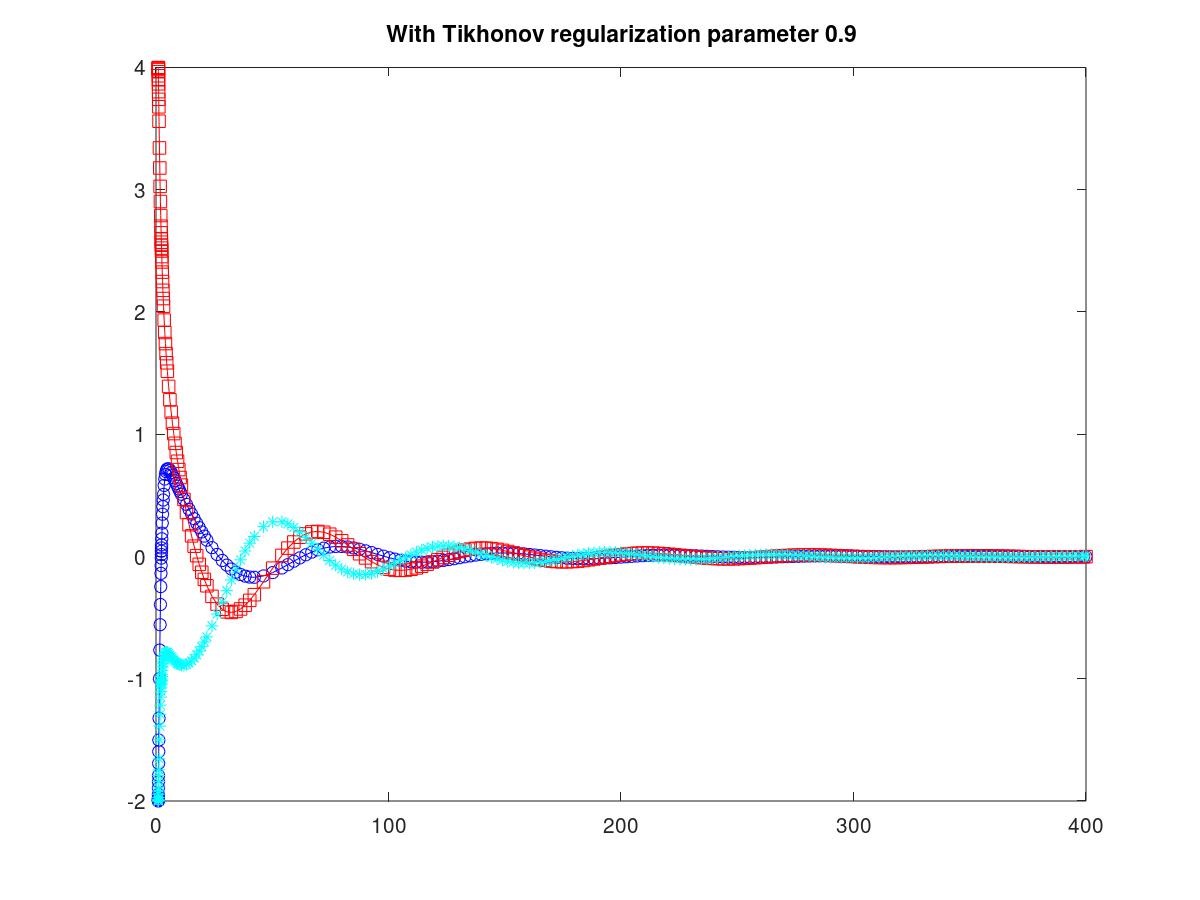}
	\includegraphics[width=0.49\textwidth]{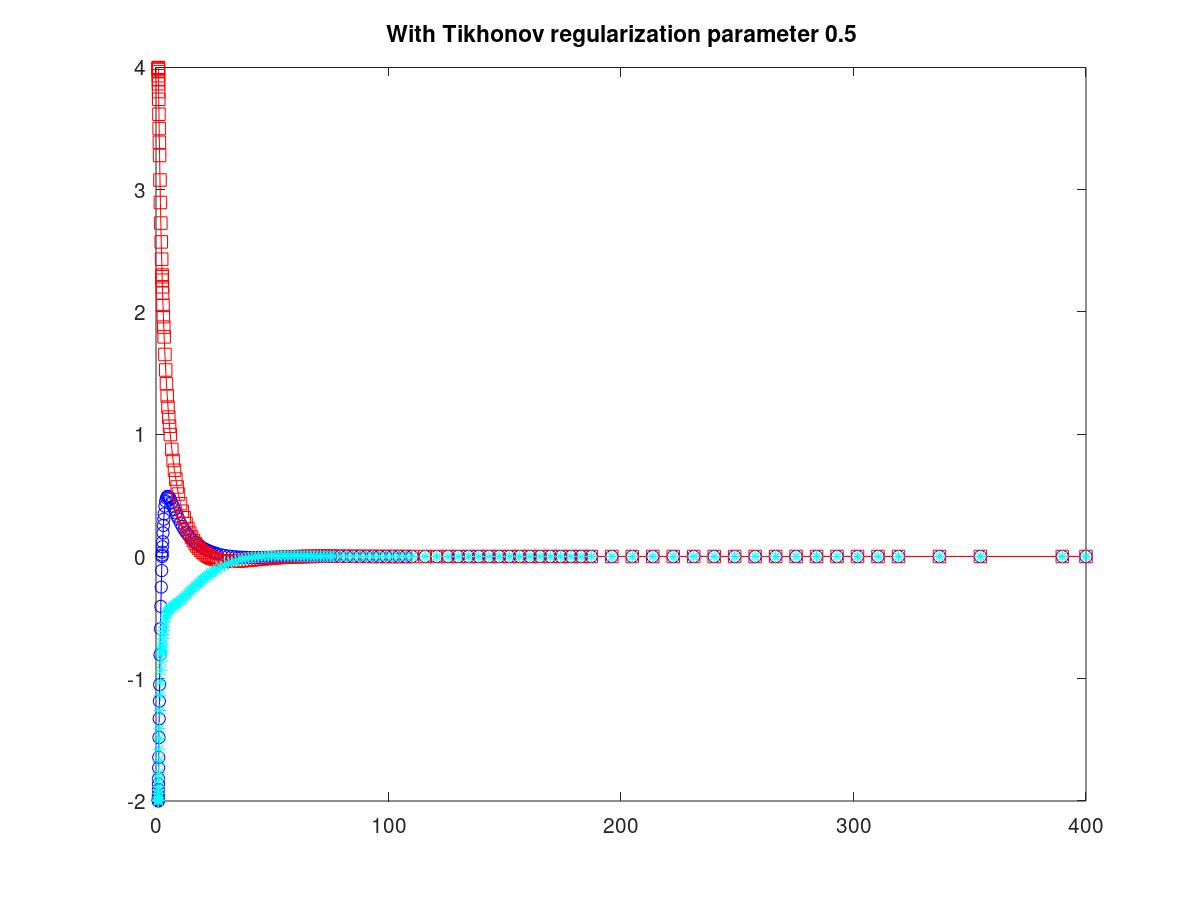}
	\includegraphics[width=0.49\textwidth]{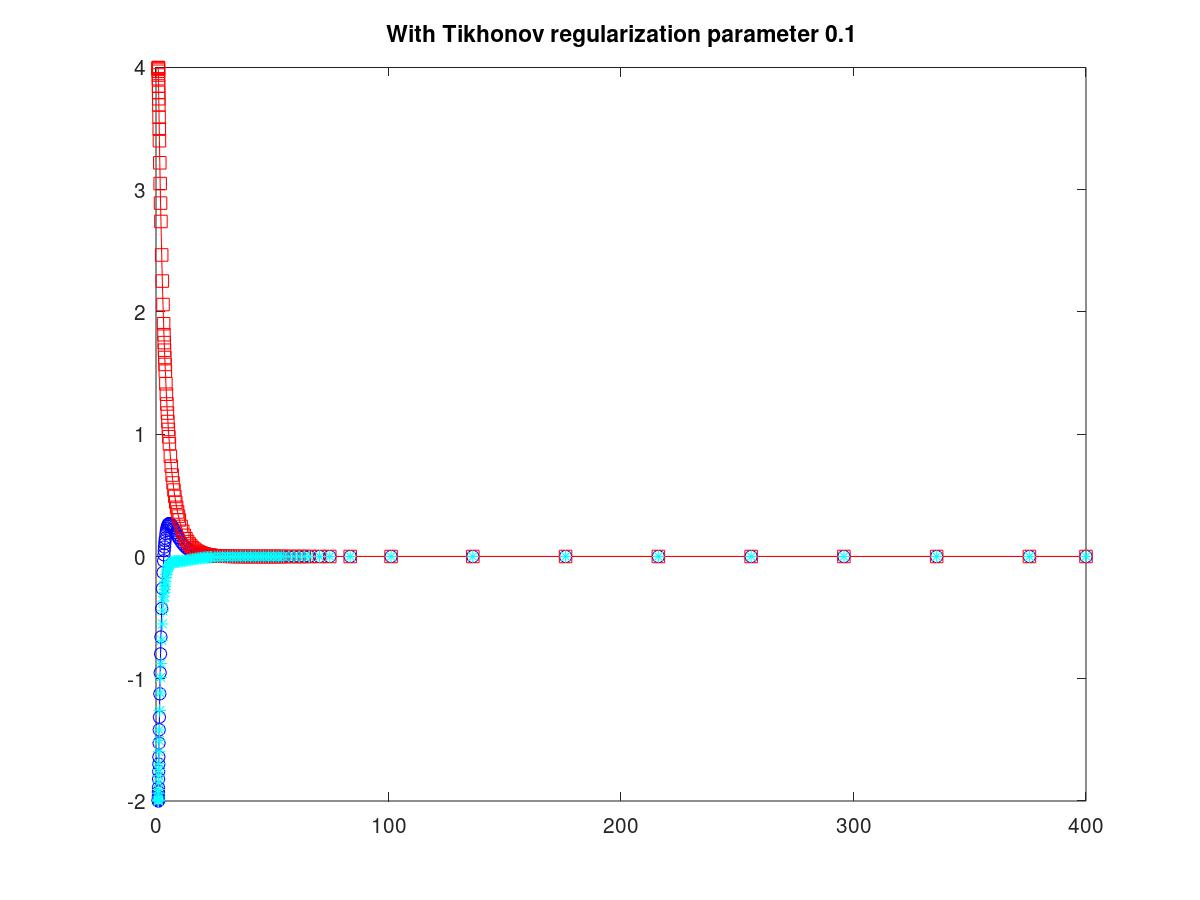}
	\caption{Trajectories of \eqref{FBFVI} for regularization parameters $\beta \in \{0,0.1,0.5,0.9\}$ and $\gamma = 0.5$}\label{FBFVI05}
\end{figure}
One observes that the unregularized trajectories are oscillating with high frequency and converge slowly to zero. As we employ the Tikhonov regularization, the oscillating behaviour flattens out and the convergence speed increases. Since the parameter $\beta$ is the exponent in the denominator of $\epsilon$, a small value of $\beta$ corresponds to a stronger impact of the Tikhonov regularization and vice versa. Hence, the two above mentioned effects are most pronounced when $\beta$ is small.
Moreover, comparing Figures \ref{FBFVI02} and \ref{FBFVI05} suggests that increasing the step size $\gamma$ results in an acceleration of the convergence behaviour (note the different time scales in Figures \ref{FBFVI02} and \ref{FBFVI05}).

\section{Conclusions}
\label{sec:conclusion}
In this paper we perturb by means of the Tikhonov regularization several dynamical systems in order to guarantee the strong convergence of their trajectories under reasonable hypotheses. First we investigate a Tikhonov regularized Krasnoselskii-Mann dynamical system and show that its trajectories strongly converge towards a minimum norm fixed point of the involved nonexpansive operator, slightly extending some recent results from the literature. As a special case, a perturbed forward-backward dynamical system with an outer Tikhonov regularization is obtained, whose trajectories strongly converge towards the minimum norm zero of the sum of a maximally monotone operator and a single-valued cocoercive operator. Making the Tikhonov regularization an inner one, by perturbing the single-valued operator and not the whole system as above, another Tikhonov regularized forward-backward dynamical system, this time with dynamic stepsizes (in contrast to the constant ones considered before) is obtained and its trajectories strongly converge towards the minimum norm zero of the mentioned sum of operators as well. Afterwards we consider an implicit forward-backward-forward dynamical system with a similar inner Tikhonov regularization of the involved single-valued operator, that is taken to be only Lipschitz continuous this time. The trajectories of this perturbed dynamical system strongly converge towards the minimum norm zero of the sum of a maximally monotone operator with the mentioned single-valued Lipschitz continuous one. These results improve previous contributions from the literature where only weak convergence of such trajectories was obtained under standard assumptions, more demanding hypotheses of uniform monotonicity or strong monotonicity being employed for deriving strong convergence. In order to illustrate our achievements we present some numerical experiments performed in {\sc Matlab} by using the {\sc ode15s} function for solving ordinary differential equations. In order to deal with the forward-backward dynamical systems we consider a split feasibility problem, while for the forward-backward-forward dynamical system we use a variational inequality. In both these situations one can note that adding a Tikhonov regularization term in the considered dynamical systems significantly influences the asymptotic behaviour of their trajectories. More precisely, while the trajectories of the unregularized dynamical systems are oscillating with high frequency and converge slowly towards some (random) solutions of the considered problems, the regularized dynamical systems generate trajectories which converge to the corresponding minimum norm solutions. Moreover, the outer regularization acts more aggressively than the inner regularization, leading to a faster convergence of the trajectories.

\section*{Acknowledgements}

The work of R.I. Bo\c{t} was supported by FWF (Austrian Science Fund), project I2419-N32.
The work of S.-M. Grad was supported by FWF (Austrian Science Fund), project M-2045, and by DFG (German Research Foundation), project GR 3367/4-1.
The work of D. Meier was supported by FWF (Austrian Science Fund), project I2419-N32, by the Doctoral Programme Vienna Graduate School on Computational Optimization (VGSCO), project W1260-N35 and by DFG (German Research Foundation), project GR3367/4-1. 
M. Staudigl thanks the COST Action CA16228 ``European Network for Game Theory'' for financial support. The authors thank Phan Tu Vuong for valuable discussions.



\bibliographystyle{plainnat}
\bibliography{refs}

\begin{thebibliography}{28}
\providecommand{\natexlab}[1]{#1}
\providecommand{\url}[1]{\texttt{#1}}
\expandafter\ifx\csname urlstyle\endcsname\relax
  \providecommand{\doi}[1]{doi: #1}\else
  \providecommand{\doi}{doi: \begingroup \urlstyle{rm}\Url}\fi

\bibitem[Abbas and Attouch(2015)]{AbbAtt15}
B.~Abbas and H.~Attouch.
\newblock Dynamical systems and forward--backward algorithms associated with
  the sum of a convex subdifferential and a monotone cocoercive operator.
\newblock \emph{Optimization}, 64\penalty0 (10):\penalty0 2223--2252, 2015.

\bibitem[Alvarez et~al.(2002)Alvarez, Attouch, Bolte, and Redont]{AABR02}
F.~Alvarez, H.~Attouch, J.~Bolte, and P.~Redont.
\newblock A second-order gradient-like dissipative dynamical system with
  {Hessian} damping. {Applications} to optimization and mechanics.
\newblock \emph{Journal des Math\'ematiques Pures et Appliqu\'ees},
  81:\penalty0 774--779, 2002.

\bibitem[Attouch and Cominetti(1996)]{AC6}
H.~Attouch and R.~Cominetti.
\newblock {A Dynamical Approach to Convex Minimization Coupling Approximation
  with the Steepest Descent Method}.
\newblock \emph{Journal of Differential Equations}, 128\penalty0 (2):\penalty0
  519--540, 1996.

\bibitem[Attouch and Czarnecki(2010)]{AC2}
H.~Attouch and M.-O. Czarnecki.
\newblock {Asymptotic behavior of coupled dynamical systems with multiscale
  aspects}.
\newblock \emph{Journal of Differential Equations}, 248\penalty0 (6):\penalty0
  1315--1344, 2010.

\bibitem[Attouch et~al.(2004)Attouch, Bolte, Redont, and Teboulle]{ABRT04}
H.~Attouch, J.~Bolte, P.~Redont, and M.~Teboulle.
\newblock Singular {Riemannian} barrier methods and gradient-projection
  dynamical systems for constrained optimization.
\newblock \emph{Optimization}, 53\penalty0 (5--6):\penalty0 435--454, 2004.

\bibitem[Aubin(1991)]{Aub91}
J.-P. Aubin.
\newblock \emph{Viability Theory}.
\newblock Birkh{\"a}user, Boston, 1991.

\bibitem[Aubin and Cellina(1984)]{AubCel84}
J.-P. Aubin and A.~Cellina.
\newblock \emph{Differential Inclusions}.
\newblock Springer, Berlin, 1984.

\bibitem[Baillon(1978)]{Bai78}
J.B. Baillon.
\newblock Un exemple concernant le comportement asymptotique de la solution du
  probl{\`e}me ${\dif u}/{\dif t}+\partial \phi(\mu)\ni 0$.
\newblock \emph{Journal of Functional Analysis}, 28\penalty0 (3):\penalty0
  369--376, 1978.

\bibitem[Baillon and Brezis(1976)]{BailBre76}
JB~Baillon and H.~Brezis.
\newblock Une remarque sur le comportement asymptotique des semigroupes non
  lin{\'e}aires.
\newblock \emph{Houston J. Math.}, 2:\penalty0 5--7, 1976.

\bibitem[Banert and Bo\c{t}(2018)]{BB8}
S.~Banert and R.~I. Bo\c{t}.
\newblock {A Forward-Backward-Forward Differential Equation and its Asymptotic
  Properties}.
\newblock \emph{Journal of Convex Analysis}, 25\penalty0 (2):\penalty0
  371--388, 2018.

\bibitem[Bauschke and Combettes(2017)]{BC2}
H.~H. Bauschke and P.~L. Combettes.
\newblock \emph{{Convex Analysis and Monotone Operator Theory in Hilbert
  Spaces}}.
\newblock Springer - CMS Books in Mathematics, 2 edition, 2017.

\bibitem[{Bo{\c t}} and Csetnek(2017)]{BC7}
R.~I. {Bo{\c t}} and E.~R. Csetnek.
\newblock {A dynamical system associated with the fixed points set of a
  nonexpansive operator}.
\newblock \emph{Journal of Dynamics and Differential Equations}, 29\penalty0
  (1):\penalty0 155--168, 2017.

\bibitem[Bo\c{t} et~al.(2018)Bo\c{t}, Csetnek, and Vuong]{BotCseVuo18}
R.I. Bo\c{t}, E.R. Csetnek, and P.T. Vuong.
\newblock The forward-backward-forward method from discrete and continuous
  perspective for pseudo-monotone variational inequalities in {H}ilbert spaces.
\newblock \emph{arXiv:1808.08084}, 2018.

\bibitem[Bolte(2003)]{BOL}
J.~Bolte.
\newblock {Continuous gradient projection method in {H}ilbert spaces}.
\newblock \emph{Journal of Optimization Theory and its Applications},
  119\penalty0 (2):\penalty0 235--259, 2003.

\bibitem[{Bruck Jr}(1974)]{BRU}
R.~E. {Bruck Jr}.
\newblock {A strongly convergent iterative solution of $0\in {U}(x)$ for a
  maximal monotone operator ${U}$ in {H}ilbert space}.
\newblock \emph{Journal of Mathematical Analysis and Applications}, 48\penalty0
  (1):\penalty0 114--126, 1974.

\bibitem[Bruck~Jr(1975)]{Bru75}
Ronald~E Bruck~Jr.
\newblock Asymptotic convergence of nonlinear contraction semigroups in hilbert
  space.
\newblock \emph{Journal of Functional Analysis}, 18\penalty0 (1):\penalty0
  15--26, 1975.

\bibitem[Cominetti et~al.(2008)Cominetti, Peypouquet, and Sorin]{CPS}
R.~Cominetti, J.~Peypouquet, and S.~Sorin.
\newblock {Strong asymptotic convergence of evolution equations governed by
  maximal monotone operators with Tikhonov regularization}.
\newblock \emph{Journal of Differential Equations}, 245\penalty0 (12):\penalty0
  3753--3763, 2008.

\bibitem[Haraux(1991)]{HAR}
A.~Haraux.
\newblock \emph{{Syst{\'e}mes Dynamicques Dissipatifs et Applications}}.
\newblock Masson, Paris, 1991.

\bibitem[Lions and Stampacchia(1967)]{LioSta67}
J.-L. Lions and G.~Stampacchia.
\newblock Variational inequalities.
\newblock \emph{Communications on Pure and Applied Mathematics}, 20:\penalty0
  493--519, 1967.

\bibitem[Lions and Mercier(1979)]{LioMer79}
P.~L. Lions and B.~Mercier.
\newblock Splitting algorithms for the sum of two nonlinear operators.
\newblock \emph{SIAM Journal on Numerical Analysis}, 16\penalty0 (6):\penalty0
  964--979, 1979.

\bibitem[Mertikopoulos and Staudigl(2018{\natexlab{a}})]{MerSta18}
P.~Mertikopoulos and M.~Staudigl.
\newblock On the convergence of gradient-like flows with noisy gradient input.
\newblock \emph{SIAM Journal on Optimization}, 28\penalty0 (1):\penalty0
  163--197, 2018/08/04 2018{\natexlab{a}}.
\newblock \doi{10.1137/16M1105682}.
\newblock URL \url{https://doi.org/10.1137/16M1105682}.

\bibitem[Mertikopoulos and Staudigl(2018{\natexlab{b}})]{MerSta18b}
Panayotis Mertikopoulos and Mathias Staudigl.
\newblock Stochastic mirror descent dynamics and their convergence in monotone
  variational inequalities.
\newblock \emph{Journal of Optimization Theory and Applications}, 179\penalty0
  (3):\penalty0 838--867, 2018{\natexlab{b}}.

\bibitem[Pardoux and R\u{a}\c{s}canu(2014)]{PR4}
E.~Pardoux and A.~R\u{a}\c{s}canu.
\newblock \emph{{Stochastic Differential Equations, Backward SDEs and Partial
  Differential Equations}}.
\newblock Springer, 2014.

\bibitem[Peypouquet and Sorin(2010)]{PeySor10}
Juan Peypouquet and Sylvain Sorin.
\newblock Evolution equations for maximal monotone operators: Asymptotic
  analysis in continuous and discrete time.
\newblock \emph{Journal of Convex Analysis}, pages 1113--1163, 2010.

\bibitem[Sontag(2013)]{SON}
E.~D. Sontag.
\newblock \emph{{Mathematical control theory: deterministic finite dimensional
  systems}}, volume~6.
\newblock Springer Science \& Business Media, 2013.
\newblock ISBN 1461205778.

\bibitem[Tseng(2000)]{TSE}
P.~Tseng.
\newblock A modified forward-backward splitting method for maximal monotone
  mappings.
\newblock \emph{SIAM Journal on Control and Optimization}, 38\penalty0
  (2):\penalty0 431--446, 2000.

\bibitem[Vilches and P{\'e}rez-Aros(2019)]{PAV}
P.~Vilches and E.~P{\'e}rez-Aros.
\newblock {Tikhonov regularization of dynamical systems associated with
  nonexpansive operators defined in closed and convex sets}.
\newblock \emph{arXiv}, 1904.05718, 2019.

\bibitem[Wibisono et~al.(2016)Wibisono, Wilson, and Jordan]{Wib16}
Andre Wibisono, Ashia~C. Wilson, and Michael~I. Jordan.
\newblock A variational perspective on accelerated methods in optimization.
\newblock \emph{Proceedings of the National Academy of Sciences}, 113\penalty0
  (47):\penalty0 E7351, 11 2016.
\newblock URL \url{http://www.pnas.org/content/113/47/E7351.abstract}.

\end{thebibliography}
\end{document}